\documentclass[12pt]{amsart}

\usepackage{titletoc}
\usepackage[utf8]{inputenc}
\usepackage[T1]{fontenc}
\usepackage{indentfirst}
\usepackage[margin=1in]{geometry} 
\usepackage{lipsum}
\usepackage{amsmath,amscd,amsthm,amssymb,mathrsfs,stmaryrd,color,tikz-cd,graphicx,tikz,mathtools,subfigure}
\tikzcdset{arrow style=tikz,diagrams={>=stealth}}
\usepackage{enumerate}
\usepackage[bookmarks, bookmarksdepth=1, colorlinks=true, linkcolor=blue, citecolor=red, urlcolor=blue,pagebackref]{hyperref}
\usepackage{subfiles}

\usetikzlibrary{matrix,arrows,decorations.pathmorphing,cd,topaths,calc,graphs}
\usepackage[all,cmtip]{xy}

\usepackage{relsize}

\newcommand{\sgn}{\operatorname{sgn}}

\renewcommand{\Re}{\operatorname{Re}}
\renewcommand{\Im}{\operatorname{Im}}

\makeatletter
\newcommand{\subalign}[1]{%
  \vcenter{%
    \Let@ \restore@math@cr \default@tag
    \baselineskip\fontdimen10 \scriptfont\tw@
    \advance\baselineskip\fontdimen12 \scriptfont\tw@
    \lineskip\thr@@\fontdimen8 \scriptfont\thr@@
    \lineskiplimit\lineskip
    \ialign{\hfil$\m@th\scriptstyle##$&$\m@th\scriptstyle{}##$\crcr
      #1\crcr
    }%
  }
}
\makeatother

\newcommand{\BA}{\mathbb{A}}

\newcommand{\BC}{\mathbb{C}}
\newcommand{\BD}{\mathbb{D}}

\newcommand{\BH}{\mathbb{H}}

\newcommand{\BP}{\mathbb{P}}
\newcommand{\BQ}{\mathbb{Q}}
\newcommand{\BR}{\mathbb{R}}

\newcommand{\BZ}{\mathbb{Z}}

\newcommand{\cA}{\mathcal{A}}

\newcommand{\cC}{\mathcal{C}}

\newcommand{\cD}{\mathcal{D}}

\newcommand{\sF}{\mathscr{F}}
\newcommand{\bG}{\mathbf{G}}

\newcommand{\cH}{\mathcal{H}}

\newcommand{\cM}{\mathcal{M}}

\newcommand{\cO}{\mathcal{O}}

\newcommand{\sP}{\mathscr{P}}

\newcommand{\cS}{\mathcal{S}}

\newcommand{\sS}{\mathscr{S}}

\newcommand{\cT}{\mathcal{T}}

\newcommand{\sT}{\mathscr{T}}

\newcommand{\bZ}{\mathbf{Z}}

\newcommand{\fa}{\mathfrak{a}}

\newcommand{\fg}{\mathfrak{g}}

\newcommand{\fk}{\mathfrak{k}}

\newcommand{\fm}{\mathfrak{m}}

\newcommand{\fn}{\mathfrak{n}}

\newcommand{\fp}{\mathfrak{p}}

\newcommand{\fu}{\mathfrak{u}}

\newcommand{\bw}{\mathbf{w}}

\newcommand{\bz}{\mathbf{z}}
\newcommand{\fz}{\mathfrak{z}}

\newcommand{\vertiii}[1]{{\left\vert\kern-0.25ex\left\vert\kern-0.25ex\left\vert #1 
    \right\vert\kern-0.25ex\right\vert\kern-0.25ex\right\vert}}

\newcommand{\rom}[1]{\uppercase\expandafter{\romannumeral #1\relax}}

\renewcommand{\tilde}[1]{\widetilde{#1}}

\DeclareMathOperator{\supp}{Supp}

\DeclareMathOperator{\End}{End}

\newcommand{\GL}{\operatorname{GL}}

\makeatletter
\newcommand{\colim@}[2]{%
  \vtop{\m@th\ialign{##\cr
    \hfil$#1\operator@font lim$\hfil\cr
    \noalign{\nointerlineskip\kern1.5\ex@}#2\cr
    \noalign{\nointerlineskip\kern-\ex@}\cr}}%
}
\newcommand{\Lim}{%
  \mathop{\mathpalette\varlim@{\leftarrowfill@\scriptscriptstyle}}\nmlimits@
}

\newcommand{\colim}{%
  \mathop{\mathpalette\varlim@{\rightarrowfill@\scriptscriptstyle}}\nmlimits@
}

\makeatother

\newtheorem{theorem}{Theorem}[section]

\newtheorem{proposition}[theorem]{Proposition}

\newtheorem{corollary}[theorem]{Corollary}

\newtheorem{lemma}[theorem]{Lemma}

\theoremstyle{definition}

\theoremstyle{remark}

\newtheorem*{rmk}{Remark}

\usepackage[new]{old-arrows}

\newcommand{\overbar}[1]{\mkern 1.5mu\overline{\mkern-1.5mu#1\mkern-1.5mu}\mkern 1.5mu}

\usepackage{amssymb}

\usepackage[utf8]{inputenc}

\title{Kakeya-Nikodym norms of Maass forms on $\rm{U}(2,1)$}
\author{Jiaqi Hou}
\address{Department of Mathematics, Louisiana State University, Baton Rouge, LA 70803, USA}
\email{jhou7@lsu.edu}
\date{}

\usepackage{hyperref}
\numberwithin{equation}{section}

\begin{document}

\begin{abstract}
    Let $\psi$ be a Hecke-Maass form with a large spectral parameter on a compact arithmetic complex hyperbolic surface. We apply the amplification method to obtain a power saving over the trivial bound for the Kakeya-Nikodym norm of $\psi$. As a consequence, we obtain power savings over the local bound of Sogge for $\|\psi\|_p$ when $2<p<10/3$.
\end{abstract}

\maketitle

{
}

\section{Introduction}
Let $M$ be a compact Riemannian manifold of dimension $n$, and let $\Delta$ be the Laplace-Beltrami operator on $M$.
If $\psi$ is a Laplace eigenfunction on $M$
satisfying  $\Delta\psi + \lambda^2\psi = 0$ with $\lambda\geq 0$ and $\|\psi \|_2 =1$,
a classical result on $\| \psi\|_p$ is due to Sogge  \cite{sogge1988concerning} (see also Avakumovi{\'c} \cite{avakumovic1956eigenfunktionen}  and Levitan \cite{levitan1952asymptotic} when $p=\infty$), 
and is
\begin{align}\label{eq:Sogge bound}
    \| \psi\|_p \ll \lambda^{\delta(p,n)}, 
\end{align}
where $\delta(p,n)$ is given by
\begin{equation*}
            \delta(p,n) = \begin{dcases}
                \frac{n-1}{4} - \frac{n-1}{2p} \;&\text{ for }\;2\leq p\leq \frac{2(n+1)}{n-1},\\
                \frac{n-1}{2} - \frac{n}{p} \;\;\;\;&\text{ for }\;\frac{2(n+1)}{n-1}\leq p\leq\infty.
            \end{dcases}
\end{equation*}
When $M$ is the round sphere, the above estimates are saturated by the zonal spherical harmonics for $p\geq \frac{2(n+1)}{n-1}$ and by the highest weight spherical harmonics for $2< p\leq \frac{2(n+1)}{n-1}$. However, it is expected that \eqref{eq:Sogge bound} can be improved under extra geometric assumptions on $M$.
For instance, there are log improvements if $M$ has nonpositive sectional curvature, by the work of B\'erard \cite{berard1977wave}, Hassell and Tacy \cite{hassell2015improvement}, and Blair and Sogge \cite{blair2018concerning,blair2019logarithmic}.

In this paper, we will focus on the problem on improving \eqref{eq:Sogge bound} in the range 
\begin{align}\label{eq: small p range}
    2< p< \frac{2(n+1)}{n-1}
\end{align}
based on the Kakeya-Nikodym bounds for eigenfunctions by Blair-Sogge \cite{blair2017refined}.
Let $\Pi$ denote the space of unit length geodesic segments in $M$.
If $\gamma\in\Pi$, we denote the $\lambda^{-1/2}$-neighborhood  of $\gamma$ in $M$ by $\sT_{\lambda^{-1/2}}(\gamma)$.
Following \cite{blair2017refined}, the Kakeya-Nikodym norm of $\psi$ is defined as
\begin{align*}
    \vertiii{\psi}_{KN} = \sup_{\gamma\in\Pi}  \| \psi|_{\sT_{\lambda^{-1/2}}(\gamma)}  \|_2 = \left(\sup_{\gamma\in\Pi}  \int_{\sT_{\lambda^{-1/2}}(\gamma)}|\psi(x)|^2 dx\right)^{1/2},
\end{align*}
which measures the concentration of
eigenfunctions on geodesic tubes. These norms were introduced by Sogge \cite{Sogge11Tohoku}. Since we are assuming the eigenfunctions to be $L^2$-normalized, we always have the trivial upper bound
\begin{align*}
    \vertiii{\psi}_{KN}\leq \|\psi\|_2 =1.
\end{align*}
The trivial bound $\vertiii{\psi}_{KN}\ll 1$ is sharp for general Riemannian manifolds, which are saturated by the highest weight spherical harmonics on the round spheres as well.

Blair and Sogge \cite{blair2017refined} showed that improvements on the trivial bound $ \vertiii{\psi}_{KN} \ll 1$ will give improvements on the $L^p$ bounds \eqref{eq:Sogge bound} for $p$ below the critical exponent,
that is, for $2(n+2)/n< p < 2(n+1)/(n-1)$, 
\begin{align}\label{eq:Blair-Sogge}
    \| \psi \|_p \ll \lambda^{\delta(p,n)} \vertiii{\psi}_{KN}^{\frac{2(n+1)}{p(n-1)} -1 }.
\end{align}
When $n\geq 3$, the above result has recently been extended by Gao, Wu, and Xi \cite{gao2025sharp} for the larger range $2(3n+1)/(3n-3)< p < 2(n+1)/(n-1)$.
See also \cite{Bourgain,Sogge11Tohoku,BS15APDE} for the related results in dimension two.
By using Toponogov’s comparison theorem from Riemannian geometry, Blair and Sogge \cite{blair2018concerning} obtained log improvements for the Kakeya–Nikodym norms when $M$ has nonpositive curvature. Using these and \eqref{eq:Blair-Sogge}, they were able to obtain log improvements over Sogge's $L^p$-bounds \eqref{eq:Sogge bound} for such manifolds when $p$ satisfies \eqref{eq: small p range}.

In the arithmetic setting, we have further improvements over the trivial bound $\vertiii{\psi}_{KN}\ll 1$.
Now we let $M$ be a compact congruence arithmetic hyperbolic surface and let $\psi$ be an $L^2$-normalized Hecke-Maass form. In \cite[Theorem 1.1]{marshall2016geodesic}, Marshall proved an $L^2$-bound for geodesic restrictions:
\begin{align*}
    \sup_{\gamma\in\Pi} \| \psi|_\gamma \|_{L^2(\gamma)}\ll_\epsilon \lambda^{3/14+\epsilon},
\end{align*}
which implies that
\begin{align*}
    \vertiii{\psi}_{KN}\ll_\epsilon \lambda^{-1/28+\epsilon}.
\end{align*}
Combining the above estimate with the main theorem of \cite{BS15APDE}, Marshall \cite[Corollary 1.2]{marshall2016geodesic} therefore obtained an improvement over the local bound $\|\psi\|_4\ll\lambda^{1/8}$ of Sogge \eqref{eq:Sogge bound}:
\begin{align*}
    \| \psi \|_4\ll_\epsilon \lambda^{1/8-1/56+\epsilon}.
\end{align*}
The corresponding result about the Kakeya-Nikodym norms of Hecke-Maass forms on compact arithmetic hyperbolic 3-manifolds was proved in \cite{hou2024restrictions}. In this paper, we will prove a similar result for Hecke-Maass forms on compact arithmetic complex hyperbolic surfaces.
All of these results were based on the method of arithmetic amplification, which was introduced by Iwaniec and Sarnak \cite{iwaniec1995norms} to study $L^\infty$-norms of Hecke-Maass forms on arithmetic hyperbolic surfaces. See e.g. \cite{blomer2019sup,blomer2016subconvexity,blomer2016sup}, for the bound of $L^\infty$-norms of Hecke-Maass forms on other groups via the method of arithmetic amplification.

\begin{theorem}\label{thm:main theorem}
    Let $X$ be a compact arithmetic congruence complex hyperbolic surface and
    let $\psi$ be an $L^2$-normalized  Hecke–Maass form on $X$ with spectral parameter $\lambda>0$ (see \S \ref{sec: notation} for precise definitions).
    Then for any $\epsilon>0$ we have
    \begin{align*}
        \vertiii{\psi}_{KN} \ll \lambda^{-1/28+ \epsilon},
    \end{align*}
    where the implied constant depends only on $X$ and $\epsilon$.
\end{theorem}

\begin{corollary}
     Let $X$ and $\psi$ be the same as in Theorem \ref{thm:main theorem}. Then we have
    \begin{align*}
        \|  \psi \|_p \ll_\epsilon  \lambda^{\delta(p,4)- (10/3p-1)/28+ \epsilon}
    \end{align*}
    if  $26/9<p<10/3$. Here the implied constant depends only on $X$, $p$ and $\epsilon$.
\end{corollary}
This corollary is obtained by directly applying Theorem \ref{thm:main theorem} and the Kakeya-Nikodym bounds \eqref{eq:Blair-Sogge}.
Moreover, we can interpolate the above bounds with $\| \psi\|_2 =1$ to prove power savings over the local $L^p$-bounds \eqref{eq:Sogge bound} for $X$ and $\psi$ with any $2<p<10/3$.

\subsection{Outline of the paper}
The paper is outlined as follows. In \S \ref{sec: notation}, we introduce the notation and the setups. The proof of Theorem \ref{thm:main theorem} is analogous to the proof of \cite[Theorem 1.2]{hou2024restrictions}. In \S \ref{sec:amplification}, we begin by introducing the integrated pretrace formula and its amplification. Note that at finite split places, the unitary group considered is isomorphic to $\GL(3)$, which allows us to use the amplifier constructed in \cite{marshall2015restrictions} at those places. After showing a counting result, we prove Theorem \ref{thm:main theorem} in \S \ref{subsec:proof of main thm} by assuming Proposition \ref{prop: bound for I}, which provides the estimates of the oscillatory integrals appearing in the geometric side of the amplified trace formula. The proof of Proposition \ref{prop: bound for I} occupies the rest of the paper from \S \ref{sec:der of A} to \S\ref{sec: bd away from spec}.

\subsection{Acknowledgements}  The author would like to thank Simon Marshall for suggesting this problem. The author was supported by NSF grant DMS-1902173 and DMS-1954479 when he was at the University of Wisconsin-Madison.

\section{Notation}\label{sec: notation}
Throughout the paper, the notation $A\ll B$ will mean that there is a positive constant $C$ such that $|A| \leq C B$, and $A\asymp B$ will mean that there are positive constants $C_1$ and $C_2$ such that $C_1B \leq A \leq C_2B$. We also use $A =O(B)$ to mean $A\ll B$. 

\subsection{Unitary groups}\label{sec:adele}
Let $F$ be a totally real number field with $\cO=\cO_F$ the ring of integers, and let $E$ be a quadratic imaginary extension of $F$ with $\cO_E$ the ring of integers. We denote by $\BA=\BA_F$ the ring of adeles of $F$, and $\BA_f$ the ring of finite adeles. We shall denote places of $E$ and $F$ by $w$ and $v$ respectively, with corresponding completions $E_w$ and $F_v$, and define $E_v=E\otimes_F F_v$. If $v$ is a finite place of $F$, we denote by $\cO_v$ the ring of integers of $F_v$, by $\varpi_v$ a uniformizer in $\cO_v$ and by $q_v$ the cardinality of the residue field for $F_v$. We denote the norm of an ideal $\fn\subset\cO$ by $\operatorname{N}(\fn)$ and denote the trace and norm from $E$ to $F$ by $\operatorname{Tr}_E$ and $\operatorname{Nm}_E$ respectively. Let $|\cdot|_v$ be the absolute value on $F_v$ for any place $v$ of $F$, and let $|\cdot|_F = \prod_v |\cdot|_v$ be the absolute value on $\BA$. Then the absolute value on $\BA_E$ is given by $|\cdot|_E = |\mathrm{Nm}_E(\cdot)|_F=\prod_w|\cdot|_w$. We denote by $\overline{\cdot}$ the conjugation on $E$ over $F$. 

Let $V$ be a 3-dimensional vector space over $E$, equipped with a nondegenerate Hermitian form $\langle\,,\,\rangle$. We suppose that $\langle\,,\,\rangle$ is linear in the first variable and conjugate-linear in the second variable, and satisfies
\begin{align*}
    \overline{\langle v_1,v_2\rangle}=\langle v_2,v_1\rangle
\end{align*}
for all $v_1,v_2\in V$. We let $\bG=\mathrm{U}(V)$ be the associated unitary group defined over $F$, i.e.,
\begin{align*}
    \bG(F)=\{g\in\GL(V)(E)\mid \langle gv_1,gv_2\rangle=\langle v_1,v_2\rangle\text{ for all }v_1,v_2\in V\}.
\end{align*}
If $v$ is a place of $F$, we shall denote $\bG(F_v)$ by $G_v$. Moreover, for one fixed archimedean place $v_0$ of $F$, we assume that $V_{v_0}=V\otimes_F F_{v_0}$ is an indefinite Hermitian space so that $G_{v_0}$ is isomorphic to the indefinite unitary group $\mathrm{U}(2,1)$. We assume that $V_v$ is definite for any archimedean place $v$ other than $v_0$. For every real place $v$ (resp. $v_0$) of $F$, we denote by $w$ (resp. $w_0$) the complex place of $E$ above it. We will assume that $F\ne\BQ$ and thus $\bG$ is an anisotropic reductive group over $F$. The center of $\bG$ is denoted by $\bZ$.

We let $K_{v_0}$ be a maximal compact subgroup of $G_{v_0}$. For any other archimedean place $v\ne v_0$, we let $K_v=G_v$, which is isomorphic to the compact unitary group $\mathrm{U}(3)$. Then $K_\infty=\prod_{v|\infty}K_v$ is a maximal compact subgroup of $G_\infty=\prod_{v|\infty}G_v$.

Let $L\subset V$ be an $\cO_E$-lattice so that $V=L\otimes_{\cO_E}E$ and then $L_v = L\otimes_{\cO}\cO_{v}$ is an $\cO_{E,v}$-lattice in $V_v=V\otimes_F F_v$.
Let $S_\infty$ be the set of the archimedean places of $F$. Let $S$ be any finite set of places of $F$ containing $S_\infty$ and the dyadic places and satisfying that for any $v\notin S$, $L_v$ is a self-dual lattice with respect to the Hermitian form. 
For each finite place $v\notin S$, we define $K_v$ to be the stabilizer of $L_v$ in $G_v$, which is a maximal compact open subgroup. For $v\in S$, we take $K_v$ to be any compact open subgroup of $G_v$ that stabilizes $L_v$. Then $K_f=\prod_{v<\infty}K_v$ is a compact open subgroup of $\bG(\BA_f)$, and we let $K=K_\infty K_f \subset \bG(\BA)$.

We define
\begin{align*}
    \sP=\{v\notin S\text{ finite places of }F\text{ so that }v\text{ is split in }E\}.
\end{align*}
For $v\in\sP$, we suppose that $v$ splits into $ww^\prime$ in $E$ with $E_v= E_w\times E_{w'} = F_v\times F_v$ and $V_v= V_w\times V_{w^\prime}$. The Hermitian form induces a nondegenerate $F_v$-bilinear pairing $B_v: V_w\times V_{w^\prime}\to F_v$ so that the Hermitian form $\langle\,,\,\rangle:(V_w\times V_{w^\prime})\times(V_w\times V_{w^\prime})\to F_v\times F_v$ is given by
\begin{align*}
    \langle (v_1,v_1^\prime),(v_2,v_2^\prime)\rangle=(B_v(v_1,v_2^\prime),B_v(v_2,v_1^\prime)).
\end{align*}
There is an $F_v\times F_v$-isomorphism $\iota_v$ from  $V_w\times V_{w^\prime}$ onto $F_v^3\times F_v^3$ that carries $B_v$ to the standard bilinear pairing on $F_v^3\times F_v^3$. By the self-dual assumption, we can assume that $L_v$ is mapped to $\cO_v^3\times\cO_v^3$ under the isomorphism $\iota_v$. Then $\iota_v$ induces the isomoprhisms $G_v\simeq \GL(3,F_v)$ and $K_v\simeq\GL(3,\cO_v)$, where the action of $\GL(3,F_v)$ is the standard multiplication on $V_w$ and is contragradient on $V_{w^\prime}$.
In the rest of the paper, we shall implicitly identify $G_v$ with $\GL(3,F_v)$ and $K_v$ with $\GL(3,\cO_v)$ via $\iota_v$ for $v\in\sP$.

\subsection{Lie groups and alegrbas}
If $A$ is a complex matrix, we denote by $A^*$ the conjugate transpose of $A$.
At $v_0$, we fix a basis for $V_{v_0}=V\otimes_E E_{v_0}$ so that the Hermitian form is given by $\langle\bz,\bw\rangle=z_1\overbar{w_3}+z_2\overbar{w_2}+z_3\overbar{w_1}$ for all $\bz=(z_i),\bw=(w_i)\in\BC^3$. Hence, we shall identify $G_{v_0}$ with the group
\begin{align*}
    \mathrm{U}(2,1)=\left\{g\in\GL(3,\BC)\mid g^*Jg=J\right\},\quad\text{ where }J=\begin{pmatrix}
        &&1\\&1&\\1&&
    \end{pmatrix}.
\end{align*}
We choose the maximal compact subgroup $K_{v_0}$ of $G_{v_0}$ to be the set of fixed points of the Cartan involution $g\in \mathrm{U}(2,1)\mapsto(g^*)^{-1}$, that is
\begin{align*}
    K_{v_0}=\left\{g\in\GL(3,\BC)\mid g^*Jg=J,g^*g=I_3\right\}\simeq \mathrm{U}(2)\times\mathrm{U}(1).
\end{align*}
For simplicity, in the rest of the paper, we will denote $G_{v_0}$, $K_{v_0}$ by $G_0$, $K_0$.

The Lie algebra $\fg=\fu(2,1)$ of $G_0$ consists of matrices $X\in\mathfrak{gl}(3,\BC)$ satisfying $X^*J+JX=0$.
The Cartan involution sends a matrix $X\in\fg$ to $-X^*$. The Lie algebra $\fk$ of $K_0$ is the 1-eigenspace of the Cartan involution. Let $\fp$ be the -1-eigenspace of the Cartan involution. A maximal abelian subspace of $\fp$ can be taken to be
\begin{align*}
    \fa=\{tH\mid t\in\BR\},\quad\text{ where }H = \begin{pmatrix}1/2& 0&0\\0&0&0 \\0 &0&-1/2 \end{pmatrix}
\end{align*}
For $t\in\BR$, we define
\begin{align}\label{eq: para for A}
     a(t):=\exp(tH) = \begin{pmatrix}e^{t/2} &&\\&1& \\&&e^{-t/2}\end{pmatrix}.
\end{align}
Then $A=\{a(t)\mid t\in\BR\}$ is the connected component of a split torus of $G_0$. We let $\epsilon\in\fa^*$ be the functional on $\fa$ sending $tH$ to $t$, so
\begin{align*}
    \Sigma=\Sigma(\fa,\fg)=\{\pm\frac{1}{2}\epsilon,\pm\epsilon\}
\end{align*}
is the set of restricted roots for $\fg$ with respect to $\fa$, and let $\Sigma^+=\{\frac{1}{2}\epsilon,\epsilon\}$ be the set of positive roots. If $\alpha\in\Sigma$, we let $\fg_\alpha\subset\fg$ be the corresponding root space. In particular,
\begin{align*}
    &\fg_{\epsilon/2}=\left\{\begin{pmatrix}
        0&z&0\\0&0&-\overline{z}\\0&0&0
    \end{pmatrix}\mid z\in\BC\right\},\\
    &\fg_{\epsilon}=\left\{\begin{pmatrix}
        0&0&i\tau\\0&0&0\\0&0&0
    \end{pmatrix}\mid \tau\in\BR\right\}.
\end{align*}
Let $m_\alpha=\dim\fg_\alpha$ be the multiplicity of the root $\alpha\in\Sigma$, so $m_{\epsilon/2}=2$, and $m_{\epsilon}=1$. We let
\begin{align*}
    \rho=\frac{1}{2}\sum_{\alpha\in\Sigma^+}m_\alpha\alpha=\epsilon
\end{align*}
be the half sum of positive roots.
We have a restricted root space decomposition $\fg=\fm\oplus\fa\oplus\sum_{\alpha\in\Sigma}\fg_\alpha$ where $\fm= Z_\fk(\fa)$ is the centralizer of $\fa$ in $\fk$. We denote the center of $\fg$ by $\fz$. Then
\begin{align*}
    \fm =\fz\oplus\BR\begin{pmatrix}
        0&&\\&i&\\&&0
    \end{pmatrix}
\end{align*}

We will identify $\fa$ and $\fa^*$ with $\BR$ via the maps $H\mapsto 1$ and $\epsilon\mapsto1$. Then the natural pairing between $\fa$ and $\fa^*$ is the multiplication on $\BR$. Let $\fn=\fg_\epsilon+\fg_{2\epsilon}$ be the nilpotent subalgebra and let $N$ be the subgroup of $G_0$ with Lie algebra $\fn$. Since the exponential map is a diffeomorphism from $\fn$ onto $N$, we shall parametrize $N$ by the pair $(z,\tau)\in\BC\times\BR$. In other words, the map
\begin{align}\label{eq: para for N}
    (z,\tau)\mapsto n(z,\tau):=\exp\left(\begin{pmatrix}
        0&\sqrt{2}z&i\tau\\0&0&-\sqrt{2}\overline{z}\\0&0&0
    \end{pmatrix} \right)=\begin{pmatrix}
        1&\sqrt{2}z&i\tau-{|z|^2}\\&1&-\sqrt{2}\overline{z}\\&&1
    \end{pmatrix}
\end{align}
is a diffeomorphism from $\BC\times\BR$ onto $N$.
We have the Iwasawa decomposition for $G_0$:
\begin{align*}
    G_0 = NAK_0.
\end{align*}
For any $g\in G_0$, we let $n(g)\in N$, $A(g)\in\BR$ and $\kappa(g)\in K_0$, defined by the decomposition $g = n(g)\exp\left(A(g)H\right) \kappa(g)$.

Let $M^\prime$ and $M$ be the normalizer and centralizer, respectively, of $A$ in $K_0$. Then the Lie algebras of $M^\prime$ and $M$ are both $\fm$, and $W=M^\prime/M\simeq \BZ/2$ is the Weyl group. We choose $w_0=\big(\begin{smallmatrix}
        &&-1\\&-1&\\-1&&
    \end{smallmatrix}\big)$ 
to be a representative for the nontrivial Weyl element. 

We equip $M_3(\BC)$ with the standard Euclidean norm as a 9-dimensional complex vector space, which we denote by $\|\cdot\|$. We obtain a positive definite norm on $\mathfrak{gl}(3,\BC)$ and then on $\fg$ from $\|\cdot\|$ under the natural restriction. We let $d(\cdot,\cdot)$ be the left-invariant metric on $G_0$ and $\GL(3,\BC)$ associate to $\|\cdot\|$.

\subsection{Complex hyperbolic plane}

If $\bz=(z_i)\in\BC^3$ with $z_3\neq0$ satisfying $\langle\bz,\bz\rangle<0$, then the image of the projection map $\BP:\bz\mapsto (z_1/z_3,z_2/z_3)\in\BC^2$ defines the \textit{Siegel domain model} for the complex hyperbolic plane $\BH^2_\BC$ with
\begin{align*}
    \BH^2_\BC=\{(z_1,z_2)\in\BC^2\mid 2\mathrm{Re}(z_1)+|z_2|^2<0\}.
\end{align*}
Given $(z_1,z_2)\in\BH^2_\BC$, let $\bz=(z_1,z_2,1)^t$ be the standard lift to $\BC^3$. For any $g\in G_0=\mathrm{U}(2,1)$, $\BP(g\bz)$ is a point in the Siegel domain model. This defines the action $g\cdot(z_1,z_2)$ of $G_0$ on $\BH^2_\BC$. This action is transitive. We define $o=(-1,0)$ to be the origin of $\BH^2_\BC$. Then the maximal compact $K_0$ is the stabilizer of $o$ in $G_0$. We equip $\BH_\BC^2$ with the Bergman metric as follows. If $\bz,\bw\in\BC^3$ so that $\BP(\bz),\BP(\bw)\in\BH_\BC^2$, then the distance $d$ between them is given by
\begin{align*}
    \cosh^2(d/2) =\frac{\langle\bz,\bw\rangle\langle\bw,\bz\rangle}{\langle\bz,\bz\rangle\langle\bw,\bw\rangle}.
\end{align*}
Note that under this metric normalization, the geodesic $a(t)\cdot o$ is parametrized by arc length, and we let $\Delta$ be the Laplace-Beltrami operator on $\BH_\BC^2$ induced by this metric.

For our purpose, it is natural to write $\BH^2_\BC$ under the Iwasawa coordinates. Since
\begin{align}\label{eq: Iwa coord action}
    n(z,\tau)a(t)\cdot o=(-e^t-{|z|^2}+i\tau,-\sqrt{2}\overline{z}),
\end{align}
if $(z_1,z_2)=n(z,\tau)a(t)\cdot o$ then
\begin{align}\label{eq: Iwa coord}
    \begin{split}
        &z=-\frac{1}{\sqrt{2}}\overline{z_2},\quad\tau=\mathrm{Im}(z_1),\quad t=\log(-\mathrm{Re}(z_1)-\frac{|z_2|^2}{2}).
    \end{split}
\end{align}
We let $\cA(z,\tau,t)$ be the distance between $n(z,\tau)a(t)\cdot o$ and the origin $o$. Equivalently, by the Cartan decomposition, $\cA(z,\tau,t)$ is the unique nonnegative number so that
\[
n(z,\tau)a(t)\in K_0 a(\cA(z,\tau,t)) K_0.
\]
Using the distance formula, it may be seen that
\begin{align*}
    \cosh^2(\cA(z,\tau,t)/2)= \frac{(e^t+1+|z|^2)^2+\tau^2}{4e^t}.
\end{align*}
and thus
\begin{align}\label{eq: dist function}
    \cosh(\cA(z,\tau,t))=\cosh(t) + \frac{1}{2}e^{-t}(|z|^4+2|z|^2+\tau^2)+|z|^2.
\end{align}

Let $dz,d\tau$ be the Lebesgue measures on $\BC$ and $\BR$ respectively, and the push-forward measure under $n(z,\tau)$ is a Haar measure on $N$, which we still denote by $dzd\tau$. An invariant measure on $\BH^2_\BC$ is given by $d\mathrm{Vol}=4e^{-2\rho(tH)}dzd\tau dt=4e^{-2t}dzd\tau dt$. Equivalently, this is the volume form given by the Bergman metric on $\BH^2_\BC$ (See e.g. \cite{goldman1999complex}). We let $dk$ be the probability Haar measure on $K_0$. We now define a Haar measure $dg$ on $G_0$ through the Iwasawa decomposition $G_0=NAK_0$. Namely, for $g=n(z,\tau)a(t)k$, we let $dg=4e^{-2t}dzd\tau dtdk$.


\subsection{Hecke algebras}\label{sec: Hecke}
For any continuous function $f$ on $\bG(\BA)$, we define $f^*(g) = \overline{f(g^{-1})}$. We define the unramified Hecke algebra $\cH^S = \bigotimes_{v\notin S}^\prime \cH_v$ to be the convolution algebra of smooth functions on $\bG(\BA^S)=\prod_{v\notin S}^\prime G_v$ that are compactly supported and bi-invariant under $K^S = \prod_{v\notin S} K_v$, and $\cH_v$ denote the space of smooth, compactly supported functions on $G_v$ that are bi-invariant under $K_v$.
We let $v\in\sP$, and $(a_1,a_2,a_3)\in\BZ^3$.
Recall that we identify $G_v$ with $\GL(3,F_v)$ and $K_v$ with $\GL(3,\cO_v)$. 
Define $K_v(a_1,a_2,a_3)\subset \GL(3,F_v)$ to be the double coset
\begin{align*}
    K_v(a_1,a_2,a_3) = K_{v}\begin{pmatrix}\varpi_v^{a_1}&&\\&\varpi_v^{a_2}&\\&&\varpi_v^{a_3}\end{pmatrix}K_{v}.
\end{align*}
We let $\Phi_v(a_1,a_2,a_3)$ be the characteristic function of $K_v(a_1,a_2,a_3)$. Given ideals $\fn_1,\fn_2,\fn_3\subset\cO$ and suppose that $\fn_i$'s are only divisible by primes in $\sP$. We let
\begin{align*}
   K(\fn_1,\fn_2,\fn_3) = \prod_{v\in\sP} K_v(\operatorname{ord}_{v}(\fn_1),\operatorname{ord}_{v}(\fn_2),\operatorname{ord}_{v}(\fn_3)) \times\prod_{v<\infty,v\notin\sP} K_v,
\end{align*}
and let $\Phi(\fn_1,\fn_2,\fn_3)\in\cH_f$ be the characteristic function of $K(\fn_1,\fn_2,\fn_3)$. We shall implicitly identify $\Phi(\fn_1,\fn_2,\fn_3)$ and $K(\fn_1,\fn_2,\fn_3)$ with their images in $\bG/\bZ$ under central integration and projection. The action of $\phi\in\cH_f$ on an automorphic function $f$ on $\bG(F)\backslash \bG(\BA)$ is given by the usual formula
\begin{align*}
    (\phi f)(x) = \int_{\bG(\BA_f)}\phi(g)f(xg) dg.
\end{align*}
Here we use the Haar measures $dg_v$ on $G_v$, which are normalized so that $K_v$ has a unit volume.

\subsection{Maass forms}\label{subsec:Maass forms}
Define $X = \bG(F)\bZ(\BA)\backslash\bG(\BA)/ K$. Because $\bG$ is anisotropic, $X$ is compact. We let $\psi\in L^2(X)$ be an eigenfunction of the Laplace-Beltrami operator $\Delta$ and the Hecke algebras $\cH_v$ for all $v\in\sP$. We let $\lambda>0$ be its spectral parameter, so that
\begin{align*}
    \Delta\psi+(1+\lambda^2)\psi=0.
\end{align*}
We assume that $\|\psi\|_{L^2(X)}=1$.

\subsection{Fourier transforms}\label{subsec:Fourier transforms}
We first recall the Helgason and Harish-Chandra transforms on $\BH_\BC^2$. We refer to \cite{helgason1984groups,helgason1994geometric} for general results. For $f\in C_c^\infty(\BH^3)$,
we let the Helgason transform of $f$ be
\begin{align*}
    \widehat{f}(s,\overline{k}):= \int_{\BH_\BC^2} f(x) \exp({(1-is)A(\overline{k}x)}) d\mathrm{Vol}(x),
\end{align*}
where $s\in\BC$ and $\overline{k} \in M\backslash K_0$.
\begin{rmk}
    In \cite{helgason1994geometric}, the second variable of $\widehat{f}$ is taken to be $b$ in the boundary $B$ of $\BH^2_\BC$. As $B$ can be identified with $K_0/M$, if $b\in K_0/M$, then $\widehat{f}$ is given by
    \begin{align*}
        \widehat{f}(s,b)= \int_{\BH_\BC^2} f(x) \exp({(1-is)A(b^{-1}x)}) d\mathrm{Vol}(x).
    \end{align*}
    It is easy to see the equivalence of our definition.
\end{rmk}
For $s\in\BC$ and $x\in\BH_\BC^2$, we denote by $\varphi_s(x)$ the spherical function on $\BH^2_\BC$ with spectral parameter $s$. We will use the integral formula by Harish-Chandra for the spherical function:
\begin{align}\label{eq:HC K int formula}
    \varphi_s(x)=\int_{K_0}\exp({(1+is)A(kx)})dk=\int_{M\backslash K_0}\exp({(1+is)A(\overline kx)}) d\overline{k}.
\end{align}
Here we let $d\overline{k}$ be the probability invariant measure on $M\backslash K_0$.
If $f$ is left $K_0$-invariant, then its Helgason transform agrees with the Harish-Chandra transform on $\BH_\BC^2$, that is
\begin{align*}
    \widehat{f}(s,\overline{k}) = \widehat{f}(s) = \int_{\BH^2_\BC} f(x) \varphi_{-s}(x)d\mathrm{Vol}(x).
\end{align*}
For $f,g\in C_c^\infty(\BH^3)$,
their convolution, which we denote by $f\times g$, is defined as the convolution of the pullbacks on the group $G_0$, i.e.,
\begin{align*}
    (f\times g)(x) := \int_{G_0} f(g\cdot o) g(g^{-1}\cdot x) dg.
\end{align*}
If moreover $g$ is assumed to be $K_0$-invariant, then
\begin{align*}
    \widehat{f\times g}(s,\overline{k}) = \widehat{f}(s,\overline{k}) \widehat{g}(s).
\end{align*}
We denote by $d\nu(s)$ the Plancherel measure for $\BH^2_\BC$ so that the Fourier inversion formula holds. Namely,
\begin{align*}
    f(x)&=\frac{1}{2}\int_{-\infty}^\infty \int_{M\backslash K_0} \widehat{f}(s,\overline{k})\exp((1+is)A(\overline{k}x)) d\nu(s) d\overline{k}\\
    &=\int_{0}^\infty \int_{M\backslash K_0} \widehat{f}(s,\overline{k})\exp((1+is)A(\overline{k}x)) d\nu(s) d\overline{k}.
\end{align*}
By the formula of Gindikin-Karpelevic, it may be seen that $d\nu(s)/ds\asymp s^3$ if $s\gg 1$. We recall the Plancherel formula on $\BH^2_\BC$. For $f,g\in C_c^\infty(\BH_\BC^2)$,
\begin{align}\label{eq: PL formula}
    \int_{\BH^2_\BC} f(x)\overline{g(x)} d\mathrm{Vol}(x) = \int_0^\infty\int_{M\backslash K_0} \widehat{f}(s,\overline{k})\overline{\widehat{g}(s,\overline{k})} d\nu(s) d\overline{k}.
\end{align}
Therefore, the Helgason transform is defined for all functions in $L^2(\BH^2_\BC)$.

\medskip
Let $\ell$ be the oriented geodesic centered at $o$ in $\BH^2_\BC$ of unit length, i.e.,
\begin{align*}
    \ell=\ell_{[-1/2,1/2]} = \{a(t)\cdot o\mid t\in[-1/2,1/2] \},
\end{align*}
and the positive orientation is the direction where $t$ increases. Any oriented geodesic of unit length in $\BH^2_\BC$ is equal to $g\cdot \ell$ for some $g \in G_0$. We define the $\lambda^{-1/2}$-tube of $\ell$ to be
\begin{align*}
    \sT = \sT_{\lambda^{-1/2}}=  \{ n(z,\tau)a(t)\cdot o \mid |z| \leq \lambda^{-1/2}, |\tau|\leq \lambda^{-1/2},|t|\leq 1/2 \}.
\end{align*}
The $\lambda^{-1/2}$-tube of $g\cdot \ell$ is defined to be $g\cdot \sT$. We shall consider the $L^2$-norm problem for the Maass form $\psi$ restricted to the tube $g\cdot \sT$.
\begin{rmk}
    Note that the tubes we are considering here and in the rest of the paper are different from the ones defined in the introduction. Nevertheless, the two kinds of the tubes are comparable as $\lambda\to\infty$.
\end{rmk}

Let $\BD$ be the open cylinder with both radius and height equal to $\lambda^{-1/2}$ centered at $0$ in $\BC\times\BR$, i.e.,
\begin{align}\label{eq:defn of disc D}
    \BD =\BD_{\lambda^{-1/2}} =  \{(z,\tau)\in\BC\times\BR\mid|z|<\lambda^{-1/2}, |\tau|<\lambda^{-1/2} \},
\end{align}
and denote by $\overline{\BD}$ its closure. We denote by $\cS(\BC\times\BR \times \BR)$ the space of Schwartz functions on the Euclidean space $\BC\times\BR \times \BR \simeq \BR^4$, and let $\cS(\overline{\BD} \times \BR)$ be the subspace consisting of Schwartz functions supported in $\overline{\BD} \times \BR$.  Let $\phi(z,\tau,t)\in \cS(\BC\times\BR\times\BR)$.  We shall consider the Fourier transform in the $t$-variable, i.e.,
\begin{align*}
    (\sF_t \phi)(z,\tau,s) := \int_{-\infty}^\infty \phi(z,\tau,t) e^{-is t} dt.
\end{align*}
We will denote the Fourier inversion by $\sF_t^{-1}$.
For $\phi \in \cS(\BC\times\BR\times\BR)$, we have the integral formula of the Fourier inversion
\begin{align*}
    \phi(z,\tau,t) = \frac{1}{2\pi} \int_{-\infty}^\infty (\sF_t\phi)(z,\tau,s) e^{ist}ds.
\end{align*}

\section{Amplification and main theorem}\label{sec:amplification}
\subsection{Amplification inequality}\label{subsec:amplification}
We fix a real-valued function $h\in C^\infty(\BR)$ of Paley-Wiener type that is nonnegative and satisfies $h(0)=1$. Define $h_\lambda^0(s) = h(s-\lambda)+h(-s-\lambda)$, and let $k_\lambda^0$ be the $K_0$-bi-invariant function on $\BH^2_\BC$ with Harish-Chandra transform $h^0_\lambda$. The Paley-Wiener theorem implies that $k_\lambda^0$ is of compact support that may be chosen arbitrarily small. Define $k_\lambda = k_\lambda^0*k_\lambda^0$, which has Harish-Chandra transform $h_\lambda = (h_\lambda^0)^2$.

Let $b \in C_c^\infty(\overline{\BD} \times (-1,1))$ be a non-negative cutoff function. If $g\in G_{0}$ and $\phi\in \cS(\BC\times\BR\times\BR)$, we define $ I(\lambda,\phi,g)$ to be the integral
\begin{align*}
    \iint_{\BC\times\BR\times \BR} \overline{b\phi(z_1,\tau_1,t_1})b\phi(z_2,\tau_2,t_2) k_\lambda (a(-t_1)n(-z_1,-\tau_1)gn(z_2,\tau_2)a(t_2)) \\
   ( 4e^{-2t_1} dz_1d\tau_1dt_1)( 4e^{-2t_2} dz_2d\tau_2dt_2).
\end{align*}
We suppose that $b$ and $h_\lambda$ are chosen so that $I (\lambda, \phi, g) = 0$ unless $d(g, e)\leq1$.
For $g_0\in G_0$, we shall study the integral on $g_0 \sT$:
\begin{align*}
    \langle\psi,b\phi \rangle_{g_0\sT}:=\int_{\BC\times\BR\times\BR} \overline{b\phi}(z,\tau,t)\psi(g_0n(z,\tau)a(t)) (4e^{-2t})dzd\tau dt.
\end{align*}

\begin{proposition}\label{prop: amp ineq}
    Suppose $\cT\in\cH^S$ and $\phi\in\cS(\BC\times\BR\times\BR)$. We have
    \begin{align}\label{eq: amp ineq}
        \left| \langle \cT\psi,b\phi\rangle_{g_0\sT}\right|^2\ll \sum_{\gamma\in \mathbf{Z}(F)\backslash\mathbf{G}(F)} \left| (\cT*\cT^*)(\gamma)I(\lambda,\phi,g_0^{-1}\gamma g_0) \right|.
    \end{align}
\end{proposition}
\begin{proof}
    Consider the function 
    \begin{align*}
        K(x,y)= \sum_{\gamma\in \mathbf{Z}(F)\backslash\mathbf{G}(F)} k_\infty (\cT*\cT^*)(x^{-1}\gamma y)
    \end{align*}
    on $X\times X$, where $k_\infty$ is a compactly supported and $K_\infty$-bi-invariant function on $G_\infty$ defined by $k_\infty(x_\infty) = k_\lambda(x_{v_0})$. One has the spectral decomposition $L^2(X) = \bigoplus_i \BC\psi_i$.
    Here $\psi_i$'s are Hecke-Maass forms on $X$ with spectral parameters $\lambda_i$, which form an orthonormal basis of $L^2(X)$ and $\psi$ is one of them. Then by \cite{selberg1956harmonic}, the kernel function $K(x,y)$ has a spectral expansion $$K(x,y) = \sum_i h_\lambda(\lambda_i)\cT\psi_i(x)\overline{\cT\psi_i(y)}.$$
    If we integrate it against $\overline{b\phi}\times b\phi$ on $g_0\sT \times g_0\sT$ and write everything under the Iwasawa coordinate \eqref{eq: Iwa coord}, we obtain
    \begin{align*}
         \sum_i h_\lambda(\lambda_i)\left| \langle \cT\psi,b\phi\rangle_{g_0\sT}\right|^2=\sum_{\gamma\in \bZ(F)\backslash\mathbf{G}(F)} (\cT*\cT^*)(\gamma)I(\lambda,\phi,g_0^{-1}\gamma g_0).
    \end{align*}
    Since we have $h_\lambda(\lambda_i)\geq 0 $ for all $i$, dropping all terms but $\psi$ completes the proof.
\end{proof}

\subsection{Construction of amplifier}
We shall use the amplifier for $\GL(3)$, which is the same as the one in \cite[\S3]{marshall2015restrictions}. Let $v\in\sP$. We construct an element $T_v\in\cH_v$ that will form part of the amplifier. By \cite[Lemma 3.1]{marshall2015restrictions}, we have the following relation in $\cH_v$:
\begin{align*}
    \Phi_v(1,0,0)*\Phi_v(1,1,0)=\Phi_v(2,1,0)+(q_v^2+q_v+1)\Phi_v(1,1,1).
\end{align*}
This implies that if we define $a(\psi,v)$ and $b(\psi,v)$ by
\begin{align*}
    \Phi_v(1,0,0)\psi=a(\psi,v)q_v\psi,\quad \Phi_v(2,1,0)\psi=b(\psi,v)q_v^2\psi,
\end{align*}
then we cannot have both $|a(\psi,v)|\leq 1/2$ and $|b(\psi,v)|\leq 1/2$. We define
\begin{align}\label{eq: construction of amplifier}
    T_v:=\begin{cases}
        \Phi_v(1,0,0)/a(\psi,v)q_v\quad&\text{ if }|a(\psi,v)|\geq 1/2,\\
        \Phi_v(2,1,0)/b(\psi,v)q_v^2\quad&\text{ otherwise.}
    \end{cases}
\end{align}
It follows that $T_v\psi=\psi$ for all $v\in\sP$. We shall need the following bound for the coefficients in the expansion of $T_vT_v^*$.

\begin{lemma}[\cite{marshall2015restrictions}, Lemma 3.2]\label{lem: bound amplifier}
    Write
    \begin{align*}
        T_vT_v^*=\sum_{a_1\geq a_2\geq a_3}\alpha(a_1,a_2,a_3)\Phi_v(a_1,a_2,a_3).
    \end{align*}
    If $\alpha(a_1,a_2,a_3)\neq0$, then we have
    \begin{align}\label{eq: bound of alpha}
        \alpha(a_1,a_2,a_3)\ll q^{a_3-a_1}
    \end{align}
    where the implied constant is absolute. Moreover, one of the three pairs of inequalities
    \begin{align}\label{eq: ineq for ai}
        \begin{split}
            -1\leq a_1,a_2,a_3\leq 2\quad &\text{ and }\quad a_1+a_2+a_3=2,\\
            -2\leq a_1,a_2,a_3\leq 2\quad &\text{ and }\quad a_1+a_2+a_3=0,\\
            -2\leq a_1,a_2,a_3\leq 1\quad &\text{ and }\quad a_1+a_2+a_3=-2,
        \end{split}
    \end{align}
    holds
\end{lemma}

\subsection{Estimates of Hecke returns}
In this section, we bound the number of isometries that map the geodesic segment close to itself. If $g\in G_0$, $\fn_1,\fn_2,\fn_3\subset\cO$ are ideals that are only divisible by primes in $\sP$, and $\kappa>0$, we define
\begin{align*}
    &\overline{\cM}(g,\fn_1,\fn_2,\fn_3,\kappa)=\{\gamma\in\bZ(F)\backslash\mathbf{G}(F)\cap K(\fn_1,\fn_2,\fn_3)\mid d(g^{-1}\gamma g,e)\leq1,d(g^{-1}\gamma g,MA)\leq\kappa\},\\
    &\overline{M}(g,\fn_1,\fn_2,\fn_3,\kappa)=|\cM(g,\fn_1,\fn_2,\fn_3,\kappa)|.
\end{align*}

We recall the notation from \S \ref{sec:adele} and \S \ref{sec: Hecke}. Let $v\in\sP$ be a finite place of $F$ with $w,w^\prime$ the places of $E$ above $v$. We have the isomorphism between $\bG(F_v)$ and $\GL(3,F_v)$ and we can identify $\bG(E_v)=\GL(V)(E_{w})\times\GL(V)(E_{w^\prime})$ with $\GL(3,F_v)\times\GL(3,F_v)$. The natural embedding $\bG(F_v)\hookrightarrow\bG(E_v)$ is then given by $g\mapsto (g,{}^tg^{-1})$ from $\GL(3,F_v)$ to $\GL(3,F_v)\times\GL(3,F_v)$.

We will identify $\End_E(V)$, $\GL(V)(E)$, respectively, with $M_3(E)$, $\GL(3,E)$ by choosing a basis for $V$ over E. We let $T_0$ be the subgroup of $\GL(3,\BC)$ generated by $MA$ and the center, so that $T_0$ is the diagonal subgroup. 

We let $S_0$ be a finite subset of $F$-places such that $S\subset S_0$, and the rings of $S_0$-integers $\cO_{S_0}$ and $\cO_{E,S_0}$ are principal ideal domains. We write  $S_{0,E}$ for the set of places of $E$ above $S_0$. We let $\sP_0 = \sP \backslash S_0$ and let $L_0=L\otimes_{\cO_E}\cO_{E,S_0}$. 

Suppose that $\fn_1,\fn_2,\fn_3\subset\cO$ are ideals that are only divisible by primes in $\sP_0$. We let 
\begin{align*}
    \sS=\sS(\fn_1,\fn_2,\fn_3) = S_\infty\cup\{v\text{ finite places of }F\mid  v|\fn_1\fn_2\fn_3 \}.
\end{align*}
Denote by $\sS_E$ the set of places in $E$ that are above $\sS$.
For $B>1$, we let $\cM(g,\fn_1,\fn_2,\fn_3,\kappa,B)$ be the set consisting of $x\in \bG(F)E^\times \subset \GL(3,E)$ such that the following conditions hold.
\begin{enumerate}[(1)]
    \item For the place $w_0$, $d(g^{-1}x_{w_0} g,\BC^\times I_3)\leq 1$ and $d(g^{-1}x_{w_0} g,T_0)\leq\kappa$.
    \item If $w\neq w_0$ is an archimedean place of $E$, then $x_w\in\mathrm{U}(3)z_w$ where $B^{-1}\leq |z_w|_{w}\leq B$. Here
    \begin{align*}
        \mathrm{U}(3) = \{g\in\GL(3,\BC)\mid g^*g =I \}
    \end{align*}
    is the standard compact unitary group.
    \item If $v\notin\sS$, then $x_v\in K_v z_v$ for some $z_v\in E_v^\times$. Here $z_v = 1$ if $v\notin S_{0}$, and $z_v$ satisfies $B^{-1}\leq|z_w|_w\leq B$ if  $v\in S_0$ with $w\in S_{0,E}$ above $v$.
    \item If $v=ww^\prime\in\sP_0$ and $w,w^\prime\in\sS_E$, we have
    \begin{align*}
        (x_w,x_{w^\prime})&\in \GL(3,\cO_{E,w})\begin{pmatrix}\varpi_v^{\mathrm{ord}_v(\fn_1)}&&\\&\varpi_v^{\mathrm{ord}_v(\fn_2)}&\\&&\varpi_v^{\mathrm{ord}_v(\fn_3)}\end{pmatrix}\GL(3,\cO_{E,w})\\
        &\times\GL(3,\cO_{E,w^\prime})\begin{pmatrix}\varpi_v^{\mathrm{ord}_v(\fn_2\fn_3)}&&\\&\varpi_v^{\mathrm{ord}_v(\fn_1\fn_3)}&\\&&\varpi_v^{\mathrm{ord}_v(\fn_1\fn_2)}\end{pmatrix}\GL(3,\cO_{E,w^\prime}).
    \end{align*}
\end{enumerate}
We define $M(g,\fn_1,\fn_2,\fn_3,\kappa,B)=|\cM(g,\fn_1,\fn_2,\fn_3,\kappa,B)|$.
The following results allow us to bound $\overline{\cM}(g,\fn_1,\fn_2,\fn_3,\kappa)$.

\begin{lemma}\label{lem: bound M by L}
    There exists a constant $B>1$ so that the following holds. If $\fn_1,\fn_2,\fn_3\subset\cO$ are ideals that are only divisible by primes in $\sP_0$, then $\overline{M}(g,\fn_1,\fn_2,\fn_3,\kappa)\leq M(g,\fn_1,\fn_2,\fn_3,\kappa,B)$.
\end{lemma}
\begin{proof}
    Suppose that $\gamma\in \overline{\cM}(g,\fn_1,\fn_2,\fn_3,\kappa)$. There exists $x^\prime \in \bG(F)$ that projects to $\gamma$. It is clear that the condition (1) at the place $w_0$ is invariant under the center elements, and, therefore, $x^\prime$ satisfies (1). Moreover, $x^\prime_v\in K_v$ for almost all places $v\notin\sS$. If $v=ww^\prime$ is a place in $\sS=\sS(\fn_1,\fn_2,\fn_3)$, we have $$(x'_w,x'_{w^\prime})\in K_v(\operatorname{ord}_{v}(\fn_1),\operatorname{ord}_{v}(\fn_2),\operatorname{ord}_{v}(\fn_3))\times K_v(-\operatorname{ord}_{v}(\fn_1),-\operatorname{ord}_{v}(\fn_2),-\operatorname{ord}_{v}(\fn_3)) .$$ By our assumption on the class number for $S_{0,E}$, there exists an absolute $B>0$ and one can find $z\in E^\times$ so that $x=zx^\prime$ satisfies the conditions (2)--(4). Because distinct elements of $\overline{\cM}(g,\fn_1,\fn_2,\fn_3,\kappa)$ have distinct elements of $\cM(g,\fn_1,\fn_2,\fn_3,\kappa,B)$ assigned to them, the result follows.
\end{proof}

\begin{lemma}\label{lem: in cubic}
    Let $\Omega\subset G_0$ be a compact set. There exists $0<C<1$, which depends on $L$ and $\Omega$, such that the following holds. If $g\in\Omega$, $n=\mathrm{Nm}(\fn_1\fn_2\fn_3)$ and $0<\kappa<1$, the set $\cM(g,\fn_1,\fn_2,\fn_3,\kappa,B)$ is contained in a (commutative) cubic \'etale algebra over $E$ in $M_3(E)$ as long as
    \begin{align}\label{eq: ineq for eps X C}
        \kappa n\leq C.
    \end{align}
\end{lemma}

\begin{proof}
    We first prove that the matrices in $\cM(g,\fn_1,\fn_2,\fn_3,\kappa,B)$ are commutative with each other. For the commutativity, it suffices to check that $xy-yx$ is the zero matrix for any $x,y\in \cM(g,\fn_1,\fn_2,\fn_3,\kappa,B)$. Suppose otherwise. Let $x=(x_{ij}),y=(y_{ij})$ be one noncommutative pair and let $A=(a_{ij})_{1\leq i,j\leq 3}=xy-yx$. By the product formula and the conditions (2)--(4), it may be seen that $|\det x|_{w_0}\asymp|\det y|_{w_0}\asymp n^3$. So $|\det (n^{-1/2}x)|_{w_0}\asymp|\det (n^{-1/2}y)|_{w_0}\asymp 1$ and $|x_{ij}|_{w_0},|y_{ij}|_{w_0}\ll n$. By the condition (1), there are diagonal matrices $z_1,z_2\in T_0$ such that $d(g^{-1}x g,z_1)\leq\kappa$ and $d(g^{-1}yg,z_2)\leq\kappa$. Thus $x = gz_1g^{-1} + O(\kappa\sqrt{n})$ and $y = gz_2g^{-1} + O(\kappa\sqrt{n})$ where the extra $\sqrt{n}$ factors come from the determinants of $x$ and $y$. Since the absolute value at the complex archimedean place $w_0$ is the square of the usual distance, $|a_{ij}|_{w_0}\ll(\kappa \sqrt{n})^2 n = \kappa^2 n ^2$ for any $i,j$. By the condition (2), $x,y$ move in a compact set, and so $|a_{ij}|_{w}\ll 1$ for any other archimedean place $w$ of $E$. Since $A\neq 0$, there is at least one $a_{ij}\neq 0$. By the conditions (3) and (4), it may be seen that $\prod_{w<\infty}|a_{ij}|_w\ll 1$. By the product formula, $1=|a_{ij}|_E\ll \kappa^2 n^2$. Hence, we get a contradiction if a condition of the form \eqref{eq: ineq for eps X C} holds for $C$ sufficiently small.

    Because $\bG$ is anisotropic over $F$, all the elements in $\bG(F)E^\times$ are semisimple by the Jordan-Chevalley decomposition, and so are elements in $\cM(g,\fn_1,\fn_2,\fn_3,\kappa,B)$. By the commutativity, they can be diagonalized simultaneously over $\BC$, which completes the proof.
\end{proof}

\begin{proposition}\label{prop: Hecke return}
    Let $\Omega\subset G_0$ be a compact set. There is a constant $0<C=C(\Omega,L)<1$ such that the following holds. 
    If $g\in\Omega$, $\fn_1,\fn_2,\fn_3\subset\cO$ are ideals divisible only by primes in $\sP_0$ with $n=\mathrm{Nm}(\fn_1\fn_2\fn_3)$, and $0<\kappa\leq Cn^{-1}$, then $\overline{M}(g,\fn_1,\fn_2,\fn_3,\kappa)\ll_{\Omega,\epsilon}n^\epsilon$.
\end{proposition}
\begin{proof}
    By Lemma \ref{lem: bound M by L}, it suffices to show that ${M}(g,\fn_1,\fn_2,\fn_3,\kappa,B)\ll n^\epsilon$. Lemma \ref{lem: in cubic} implies that there is $C>0$ such that if $\kappa\leq Cn^{-1}$, then ${\cM}(g,\fn_1,\fn_2,\fn_3,\kappa,B)$ must be contained in a cubic \'etale algebra $H$ over $E$ in $M_3(E)$. Note that the determinant on $M_3(E)$ restricted to $H$ agrees with the norm map from $H$ to $E$, which we denote by $\mathrm{N}_{H/E}$. Those $x\in\cM(g,\fn_1,\fn_2,\fn_3,\kappa,B)$ satisfies the conditions:
    \begin{enumerate}[(1)]
        \item $|\mathrm{N}_{H/E}(x)|_{w_0}\asymp n^3$, and the image of $x$ under any archimedean embedding  of $H$ above $w_0$ is $\ll n$;
        \item If $w\neq w_0$ is an archimedean place, then the image of $x$ under any archimedean embedding of $H$ above $w$ is $\ll 1$;
        \item If $w\notin\sS_E$ is nonarchimedean, then 
        \begin{align*}
            |\mathrm{N}_{H/E}(x)|_{w}\begin{cases}
                =1,\quad\text{ if }w\notin S_{0,E},\\
                \asymp1,\quad\text{ if }w\in S_{0,E};\\
            \end{cases}
        \end{align*}
        \item If $v=ww^\prime\in\sP_0$ and $w,w^\prime\in\sS_E$, then $|\mathrm{N}_{H/E}(x)|_{w}=|n|_w$ and $|\mathrm{N}_{H/E}(x)|_{w'}=|n|_{w'}^2$.
    \end{enumerate}
    The amount of numbers in $H$ satisfying the above conditions (1)--(4) may easily be seen to be $\ll n^\epsilon$, uniformly in $H$.
\end{proof}

\subsection{Proof of main theorem}\label{subsec:proof of main thm}
To prove Theorem \ref{thm:main theorem}, it suffices to bound the $L^2$-norm of $b(z,\tau,t)\psi(g_0n(z,\tau)a(t))$ for all $g_0$ moving in a compact set $\Omega\subset G_0$. Let $0<\epsilon'\leq 10^{-6}$, and let $\beta$ be a parameter satisfying $\lambda^{\epsilon'}\leq\beta\leq\lambda^{1-\epsilon'}$. We shall need the following bounds for $I(\lambda,\phi,g)$.

\begin{proposition}\label{prop: bound for I}
    Suppose that $g\in G_0$ with $d(g,e)\leq1$, $\lambda^{\epsilon'}\leq\beta\leq\lambda^{1-\epsilon'}$ and $\phi \in \cS(\overline{\BD}\times\BR) $  satisfies $\|\phi \|_2 = 1$. 
    \begin{enumerate}[(1)]
        \item If $\supp(\sF_t\phi) \subset \overline{\BD}\times (\BR\backslash \pm[\lambda-\beta, \lambda+\beta])$, then we have
        \begin{align}\label{eq: bd for I(e)}
            I(\lambda,\phi,e)\ll_\epsilon \beta^{-1/2+\epsilon}.
        \end{align}
        \item Fix $\epsilon_0>0$. If $\supp(\sF_t\phi) \subset \overline{\BD}\times (\pm[\lambda-\beta,\lambda+\beta])$, and $g$ satisfies $d(g,MA)\geq \lambda^{-1/2+\epsilon_0}\beta^{1/2}$, then we have
        \begin{align}\label{eq: bd for I(g)}
            I(\lambda,\phi,g)\ll_{\epsilon_0,A} \lambda^{-A}.
        \end{align}
    \end{enumerate}
\end{proposition}

We shall prove the bound \eqref{eq: bd for I(g)}  in \S \ref{sec: bd near the spectrum}, and \eqref{eq: bd for I(e)} in \S \ref{sec: bd away from spec}. Using these with the amplification method, we may prove the following two bounds for $\langle \psi,b\phi \rangle_{g_0\sT}$. We can therefore prove Theorem \ref{thm:main theorem} by combining the following two lemmas with $\beta=\lambda^{1/7}$.

\begin{lemma}\label{lem: bound 1-Pi}
    If $g_0\in\Omega$, $\lambda^{\epsilon'}\leq\beta\leq\lambda^{1-\epsilon'}$ and $\phi \in \cS(\overline{\BD}\times\BR) $  satisfies $\|\phi \|_2 = 1$ and $\supp(\sF_t\phi) \subset \overline{\BD}\times (\BR\backslash \pm[\lambda-\beta, \lambda+\beta])$,  we have $\langle \psi,b\phi \rangle_{g_0\sT} \ll_\epsilon \beta^{-1/4+\epsilon}$.
\end{lemma}

\begin{proof}
    In Proposition \ref{prop: amp ineq}, we can choose $\cT \in \cH_f$  to be the characteristic function of a sufficiently small open subgroup of $\bG(\BA_f)/\bZ(\BA_f)$, then only the identity element $\gamma=e$ will make a nonzero contribution to the sum in \eqref{eq: amp ineq}. This gives
    \begin{align*}
        \left| \langle \psi,b\phi\rangle_{g_o\sT}\right|^2\ll\left| I(\lambda,\phi,e) \right|.
    \end{align*}
    The result follows from Proposition \ref{prop: bound for I} (1).
\end{proof}

\begin{lemma}\label{lem: bound Pi}
    If $g_0\in\Omega$, $\lambda^{\epsilon'}\leq\beta\leq\lambda^{1-\epsilon'}$ and $\phi \in \cS(\overline{\BD}\times\BR) $  satisfies $\|\phi \|_2 = 1$ and $\supp(\sF_t\phi) \subset \overline{\BD}\times (\pm[\lambda-\beta,\lambda+\beta])$, we have $\langle \psi,b\phi \rangle \ll_\epsilon \lambda^{-1/24+\epsilon}\beta^{1/24}$.
\end{lemma}

\begin{proof}
    The proof here is similar to the proof of \cite[Proposition 4.1]{marshall2015restrictions}. Let $1\leq N\leq \lambda$ be an integer to be chosen later, and define $\cT$ to be the Hecke operator
    \begin{align*}
        \cT=\sum_{v\in\sP_N} T_v
    \end{align*}
    where $\sP_N=\{v\in\sP_0\mid q_v\leq N\}$ and $T_v$ is as in \eqref{eq: construction of amplifier}.  We have
    \begin{align*}
        \cT\cT^*=\sum_{v_1,v_2\in\sP_N, v_1\neq v_2} T_{v_1}T_{v_2}^*+\sum_{v\in\sP_N} T_{v}T_{v}^*.
    \end{align*}
    We may bound
    \begin{align*}
        \sum_{\gamma\in \mathbf{Z}(F)\backslash\mathbf{G}(F)} \sum_{v_1,v_2\in\sP_N, v_1\neq v_2}\left| T_{v_1}T^*_{v_2}(\gamma)I(\lambda,\phi,g_0^{-1}\gamma g_0) \right|.
    \end{align*}
    by the sum of four expressions, the first of which is
    \begin{align}\label{eq: first expression}
        \sum_{v_1,v_2\in\sP_N, v_1\neq v_2}\sum_{\gamma\in \mathbf{Z}(F)\backslash\mathbf{G}(F)} \frac{1}{q_{v_1}q_{v_2}}\Phi(v_1v_2,v_1,1)(\gamma)\left| I(\lambda,\phi,g_0^{-1}\gamma g_0) \right|.
    \end{align}
    If we fix $\epsilon_0>0$, the bound \eqref{eq: bd for I(g)} and our assumption that $N\leq\lambda$ allow us to restrict the sum over $\gamma$ to $\overline\cM(g,v_1v_2,v_1,1,\lambda^{-1/2+\epsilon_0}\beta^{1/2})$. Let $C>0$ be the constant as in Proposition \ref{prop: Hecke return}. If we assume that $N\leq (C\lambda^{1/2-\epsilon_0}\beta^{-1/2})^{1/6}$, we may combine Proposition \ref{prop: Hecke return}, and the trivial bound $I(\lambda,\phi,g_0^{-1}\gamma g_0)\ll1$ to obtain
    \begin{align*}
        \sum_{\gamma\in \mathbf{Z}(F)\backslash\mathbf{G}(F)} \Phi(v_1v_2,v_1,1)(\gamma)\left| I(\lambda,\phi,g_0^{-1}\gamma g_0) \right|\ll_\epsilon N^\epsilon.
    \end{align*}
    Therefore, \eqref{eq: first expression} is bounded by
    \begin{align}\label{eq: first expression bd}
        \ll_\epsilon N^\epsilon  \sum_{v_1,v_2\in\sP_N, v_1\neq v_2}\frac{1}{q_{v_1}q_{v_2}}\ll_\epsilon N^\epsilon.
    \end{align}

    The other three expressions we must consider are similar, but with $\Phi(v_1v_2,v_1,1)/q_{v_1}q_{v_2}$ replaced with $\Phi(v_1^2v_2,v_1,1)/q_{v_1}^2q_{v_2}$, $\Phi(v_1v_2^2,v_1v_2,1)/q_{v_1}q_{v_2}^2$ and $\Phi(v_1^2v_2^2,v_1v_2,1)/q^2_{v_1}q^2_{v_2}$, respectively. We may bound them in the same way using our assumption that $N\leq (C\lambda^{1/2-\epsilon_0}\beta^{-1/2})^{1/6}$.

    The analysis of the $T_vT_v^*$ terms is similar. Lemma \ref{lem: bound amplifier} implies that we must consider terms of the form $\Phi(v^{a},v^{b},1)/q_v^a$ with $a\geq b 
    \geq0$, and the inequalities \eqref{eq: ineq for ai} imply that $a+b\leq 6$. Arguing as above using the bound $N\leq (C\lambda^{1/2-\epsilon_0}\beta^{-1/2})^{1/6}$ gives 
    \begin{align*}
        \sum_{\gamma\in \mathbf{Z}(F)\backslash\mathbf{G}(F)} \left| T_vT_v^*(\gamma)I(\lambda,\phi,g_0^{-1}\gamma g_0) \right|\ll_\epsilon N^\epsilon.
    \end{align*}
    If we sum over $v\in\sP_N$ and combine this with \eqref{eq: first expression bd}, we obtain
    \begin{align*}
        \sum_{\gamma\in \mathbf{Z}(F)\backslash\mathbf{G}(F)} \left| \cT\cT^*(\gamma)I(\lambda,\phi,g_0^{-1}\gamma g_0) \right|\ll_\epsilon N^{1+\epsilon},
    \end{align*}
    and Proposition \ref{prop: amp ineq} gives
    \begin{align*}
        N^{1-\epsilon}\langle\psi,b\phi\rangle\ll_\epsilon \langle\cT\psi,b\phi\rangle\ll_\epsilon N^{1/2+\epsilon}
    \end{align*}
    Choosing $N= (C\lambda^{1/2-\epsilon_0}\beta^{-1/2})^{1/6}$ completes the proof.
\end{proof}

\section{Derivatives of the Iwasawa  projection}\label{sec:der of A}
In this section, we provide several preliminary calculations that are crucial to the estimates of oscillatory integrals in \S \ref{sec: bd near the spectrum}. For $g\in G_0$, let $\Phi_g: K_0\to K_0$ be the map sending $k$ to $\kappa(kg)$, i.e $kg\in NA\Phi_g(k)$. 
\begin{lemma}\label{lem: change of variable}
    $\Phi$ is a smooth group action of $G_0$ on $K_0$ from right. Moreover, $\Phi_g$ induces a diffeomorphism from $M\backslash K_0$ onto itself.
\end{lemma}
\begin{proof}
    The smoothness of the Iwasawa decomposition implies that $\Phi_g$ is smooth and depends smoothly on $g$. We identify $K_0$ with the quotient $NA\backslash G_0$ via the Iwasawa decomposition. Then $\Phi_g$ is obtained by composing the diffeomorhpism $K_0\to NA\backslash G_0$, the right multiplication by $\cdot g:NA\backslash G_0\to NA\backslash G_0$ and the diffeomorphism $ NA\backslash G_0\to K_0$.
 It remains to show $\Phi_{gh} = \Phi_h\circ\Phi_g$ for $g,h\in G_0$. By definition
     \begin{align*}
         kgh\in NA\Phi_{gh}(k)\\
         kg\in NA\Phi_g(k),
     \end{align*}
     which implies
     \begin{align*}
         NA\Phi_g(k) h = NA\Phi_{gh}(k),\\
         \Phi_g(k) h \in NA\Phi_{gh}(k),
     \end{align*}
     that is $\Phi_h\left(  \Phi_g(k) \right) = \Phi_{gh}(k)$. The second statement follows from the fact that $M$ normalizes $NA$.
\end{proof}

\begin{lemma}\label{lem: spliting A}
    Let $y,z\in G_0$ and let $k\in K_0$. Then we have
    \begin{align*}
        A(ky^{-1}z) = A\left(\Phi_{y^{-1}}(k)z\right)- A\left(\Phi_{y^{-1}}(k)y\right).
    \end{align*}
\end{lemma}
\begin{proof}
    Let $ky^{-1} = na\Phi_{y^{-1}}(k)$. We have
    \begin{align*}
        A(ky^{-1}z) &= A\left( na\Phi_{y^{-1}}(k)z\right)\\
        &=A(na) + A\left( \Phi_{y^{-1}}(k)z\right)\\
        &=A\left( \Phi_{y^{-1}}(k)z\right) - A\left((na)^{-1}k\right)\\
        &=A\left( \Phi_{y^{-1}}(k)z\right) - A\left(\Phi_{y^{-1}}(k)y\right).
    \end{align*}
\end{proof}

\subsection{Cayley transform}
It is simpler to describe the maximal compact subgroup of the unitary group of signature $(2,1)$ if the Hermitian form is diagonal. The Cayley transform is given by the matrix
\begin{align*}
    C=\frac{1}{\sqrt{2}}\begin{pmatrix}
        1&0&1\\0&\sqrt{2}&0\\1&0&-1
    \end{pmatrix}.
\end{align*}
Observe that $C^{-1}=C$ and 
\begin{align*}
    J_1:=CJC=\begin{pmatrix}
    1&&\\&1&\\&&-1
\end{pmatrix}.
\end{align*}
We denote the unitary Lie group with the Hermitian matrix $J_1$ by $G_1$:
$$G_1=\mathrm{U}(J_1)=\{g\in\GL(3,\BC)\mid g^*J_1g=J_1\}.$$ Let $K_1$ be the subgroup of $G_1$ consisting of elements satisfying $g^*g=I$, which is a maximal compact subgroup. It may be seen that 
\begin{align*}
    K_1=\left\{ \begin{pmatrix}
        g&\\&h
    \end{pmatrix}\mid g\in\mathrm{U}(2),h\in\mathrm{U}(1)  \right\}\simeq \mathrm{U}(2)\times\mathrm{U}(1).
\end{align*}
Via the Cayley transform, i.e. conjugating by $C$, we get the isomorphism $\cC:G_0\to G_1$, which sends $K_0$ onto $K_1$. Notice that elements in $M$ are invariant under the Cayley transform, and we note that $M$ is a normal subgroup in both $K_0$ and $K_1$.

    The composition of the homomorphism from the 3-sphere $S^3\simeq \mathrm{SU}(2)$ to $K_1$ defined by
    \begin{align}\label{eq: rep of K1/M}
        \begin{pmatrix}
            \alpha&\beta\\-\overline{\beta}&\overline{\alpha}
        \end{pmatrix}\mapsto \begin{pmatrix}
            \alpha&\beta&0\\-\overline{\beta}&\overline{\alpha}&0\\0&0&1
        \end{pmatrix},\qquad\left(\alpha,\beta\in\BC,\,|\alpha|^2+|\beta|^2=1 \right),
    \end{align}
    with the natural quotient map $K_1\to M\backslash K_1$ is a diffeomorphism from $S^3$ onto $M\backslash K_1$.
Composing the map \eqref{eq: rep of K1/M} with the Cayley transform, we obtain the following injective homomorphism from $\mathrm{SU}(2)$ to $K_0$: 
\begin{align}\label{eq: isom from SU(2) to K0/M}
    \begin{pmatrix}
            \alpha&\beta\\-\overline{\beta}&\overline{\alpha}
        \end{pmatrix}\mapsto k(\alpha,\beta):=\begin{pmatrix}
            \frac{\alpha+1}{2}&\frac{\beta}{\sqrt{2}}&\frac{\alpha-1}{2}\\-\frac{\overline{\beta}}{\sqrt{2}}&\overline{\alpha}&-\frac{\overline{\beta}}{\sqrt{2}}\\\frac{\alpha-1}{2}&\frac{\beta}{\sqrt{2}}&\frac{\alpha+1}{2}
        \end{pmatrix},\qquad\left(\alpha,\beta\in\BC,\,|\alpha|^2+|\beta|^2=1 \right),
\end{align}
which induces a diffeomorphism from $\mathrm{SU}(2)$ onto $M\backslash K_0$. 
For doing explicit calculations, we choose the following basis for $\mathfrak{su}(2)$:
\begin{align}\label{eq: basis for su(2)}
    X_1 = \begin{pmatrix}0&i\\i&0\end{pmatrix}, \quad X_2 = \begin{pmatrix}0&-1\\1&0\end{pmatrix}, \quad X_3 = \begin{pmatrix}i&0\\0&-i\end{pmatrix}.
\end{align}
The corresponding vectors in $\fk$, via the Lie algebra map induced by \eqref{eq: isom from SU(2) to K0/M}, are
\begin{align}\label{eq: basis for k/m}
    \tilde{X_1} = \frac{1}{\sqrt{2}}\begin{pmatrix}0&i&0\\i&0&i\\0&i&0\end{pmatrix}, \quad \tilde{X_2} =\frac{1}{\sqrt{2}} \begin{pmatrix}0&-1&0\\1&0&1\\0&-1&0\end{pmatrix}, \quad \tilde{X_3} = \frac{1}{2}\begin{pmatrix}i&0&i\\0&-2i&0\\i&0&i\end{pmatrix}.
\end{align}
Given $(r_1,r_2,r_3)\in\BR^3$, if we let $r=(r_1^2+r_2^2+r_3^2)^{1/2}$, then it may be seen that
\begin{align}\label{eq: exp(ri)}
    \exp(r_1X_1+r_2X_2+r_3X_3)=\begin{pmatrix}
        \cos(r)+ir^{-1}\sin(r)r_3& r^{-1}\sin(r)(ir_1-r_2)\\
        r^{-1}\sin(r)(ir_1+r_2)&\cos(r)-ir^{-1}\sin(r)r_3
    \end{pmatrix}.
\end{align}
Let $U_{\fk/\fm}$ be the "unit ball" centered at 0 in the tangent space of $B$ defined by
\begin{align}\label{eq: defn of Uk/m}
    U_{\fk/\fm}:=\left\{r_1\tilde{X_1}+r_2\tilde{X_2}+r_3\tilde{X_3}\in\fk\mid r_1^2+r_2^2+r_3^2=1  \right\}.
\end{align}
Hence, from \eqref{eq: isom from SU(2) to K0/M} and \eqref{eq: exp(ri)}, it may be seen that 
\begin{align}\label{eq: exp ri tilde}
    \exp(r_1\tilde{X_1}+r_2\tilde{X_2}+r_3\tilde{X_3})=k(\cos(r)+ir^{-1}\sin(r)r_3, r^{-1}\sin(r)(ir_1-r_2)),
\end{align}
where $r=(r_1^2+r_2^2+r_3^2)^{1/2}$. 


\subsection{Explicit formula of $A(kna)$}
In this section, we return to the Siegel domain model and calculate the Iwasawa $A$-projection using some explicit parametrizations. 

Given $z\in\BC$, $\tau,t\in\BR$ and $\alpha,\beta\in\BC$ satisfying $|\alpha|^2+|\beta|^2=1$, we have the parametrizations $a(t)$, $n(z,\tau)$ and $k(\alpha,\beta)$ for $A$, $N$ and $K_0/M$ by \eqref{eq: para for A}, \eqref{eq: para for N} and \eqref{eq: isom from SU(2) to K0/M} respectively. From \eqref{eq: Iwa coord action}, the standard lift of $n(z,\tau)a(t)\cdot o\in\BH^2_\BC$ to $\BC^3$ is $(z_1,z_2,1)^t$ where $z_1=-e^t-{|z|^2}+i\tau$ and $z_2=-\sqrt{2}\overline{z}$. Therefore,
\begin{align}\label{eq: action of k}
   \begin{pmatrix}
            \frac{\alpha+1}{2}&\frac{\beta}{\sqrt{2}}&\frac{\alpha-1}{2}\\-\frac{\overline{\beta}}{\sqrt{2}}&\overline{\alpha}&-\frac{\overline{\beta}}{\sqrt{2}}\\\frac{\alpha-1}{2}&\frac{\beta}{\sqrt{2}}&\frac{\alpha+1}{2}
        \end{pmatrix}\cdot \begin{pmatrix}
        z_1\\z_2\\1
    \end{pmatrix}=\begin{pmatrix}
        \frac{\alpha+1}{2}z_1+\frac{\beta}{\sqrt{2}}z_2+\frac{\alpha-1}{2}\\
        -\frac{\overline{\beta}}{\sqrt{2}}z_1+\overline{\alpha}z_2-\frac{\overline{\beta}}{\sqrt{2}}\\
        \frac{\alpha-1}{2}z_1+\frac{\beta}{\sqrt{2}}z_2+\frac{\alpha+1}{2}
    \end{pmatrix}
\end{align}
and \eqref{eq: Iwa coord} provide that $\exp(A(k(\alpha,\beta)n(z,\tau)a(t)))$ is equal to
\begin{align*}
    &-\mathrm{Re}\left(\frac{\frac{\alpha+1}{2}z_1+\frac{\beta}{\sqrt{2}}z_2+\frac{\alpha-1}{2}}{\frac{\alpha-1}{2}z_1+\frac{\beta}{\sqrt{2}}z_2+\frac{\alpha+1}{2}} \right)-\frac{1}{2}\left| \frac{-\frac{\overline{\beta}}{\sqrt{2}}z_1+\overline{\alpha}z_2-\frac{\overline{\beta}}{\sqrt{2}}}{\frac{\alpha-1}{2}z_1+\frac{\beta}{\sqrt{2}}z_2+\frac{\alpha+1}{2}}  \right|^2\\
    =&-\frac{1}{2}\left| \frac{\alpha-1}{2}z_1+\frac{\beta}{\sqrt{2}}z_2+\frac{\alpha+1}{2} \right|^{-2}\left(  z_1+\overline{z_1}+|z_2|^2\right)\\
    =&\left| \frac{\alpha-1}{2}(-e^t-|z|^2+i\tau)-\beta \overline{z}+\frac{\alpha+1}{2} \right|^{-2} e^t
\end{align*}
In summary, we get:
\begin{lemma}\label{lem: explicit formula for A}
    Given $g\in G_0$ with the Iwasawa decomposition $g=na(t_0)mk(\alpha,\beta)$ where $n\in N$, $a(t_0)\in A$, $m\in M$ and $k(\alpha,\beta)$ is given by \eqref{eq: isom from SU(2) to K0/M}, we have, for $z\in\BC$ and $\tau,t\in\BR$,
    \begin{align}\label{eq: explicit formula for iwasawa A proj}
        A(gn(z,\tau)a(t)) = t_0 + t - \log\left(\left| \frac{\alpha-1}{2}(-e^t-|z|^2+i\tau)-\beta \overline{z}+\frac{\alpha+1}{2} \right|^{2}\right).
    \end{align}
\end{lemma}

\subsection{Derivative with respect to \texorpdfstring{$NA$}{NA}}
Fix $g=na(t_0)mk(\alpha,\beta)\in G_0$  where $n\in N$, $a(t_0)\in A$, $m\in M$ and $k(\alpha,\beta)$ is given by \eqref{eq: isom from SU(2) to K0/M}.
Taking the partial derivative of \eqref{eq: explicit formula for iwasawa A proj} with respect to $t$, we obtain
\begin{align}\label{eq: explicit formula for par A/par t}
    \left.\frac{\partial}{\partial t}A(ga(t)) \right|_{t=0} = \mathrm{Re}(\alpha).
\end{align}
If we write $z=x+iy$ with $x,y\in\BR$, then the partial derivatives with respect to $x,y$ and $\tau$ are
\begin{align}
    \left.\frac{\partial}{\partial x}A(gn(x,0)) \right|_{x=0} = 2\mathrm{Re}(\beta),\label{eq: explicit formula for par A/par x}\\
    \left.\frac{\partial}{\partial y}A(gn(iy,0)) \right|_{y=0} = 2\mathrm{Im}(\beta),\label{eq: explicit formula for par A/par y}\\
    \left.\frac{\partial}{\partial \tau}A(gn(0,\tau)) \right|_{\tau=0} = \mathrm{Im}(\alpha).\label{eq: explicit formula for par A/par tau}
\end{align}

We shall also need the following uniformization lemma for the function $A$.

\begin{lemma}\label{lem: uniformazation lemma}
    Let $\cD\subset\BC\times\BR\times\BR$ be a bounded open set. There exists $\delta,\sigma>0$ depending on $\cD$, and a real analytic function
    \begin{align*}
        \xi:(-\delta,\delta)\times U_{\fk/\fm} \times \cD\to\BR
    \end{align*}
    such that
    \begin{align*}
        \frac{\partial}{\partial t} A(\exp(rX)n(z,\tau)a(t)) = 1-r^2\xi(r,X,z,\tau,t),
    \end{align*}
    and
    \begin{align}\label{eq: xi estimate}
    \begin{split}
        |\xi(r,X,z,\tau,t)| 
        &\geq \sigma,\\
        \left|\frac{\partial^n\xi}{\partial t^n} (r,X,z,\tau,t)\right|
        &\ll_{\cD,n}1,
        \end{split}
    \end{align}
    for $r\in(-\delta,\delta)$, $X\in U_{\fk/\fm}$ and $(z,\tau,t)\in\cD$. 
\end{lemma}

\begin{proof}
    Let $\delta_0>0$ be a small number so that the exponential map from the set
    \begin{align*}
        U_{\fk/\fm}(\delta_0) := \{ rX\mid r\in (-\delta_0,\delta_0), X\in U_{\fk/\fm}     \}
    \end{align*}
    onto its image in $M\backslash K_0$ is a diffeomorphism. We define the functions 
    \begin{align*}
        \alpha,\,\beta:(-\delta_0,\delta_0) \times U_{\fk/\fm} \times \cD \to \BC
    \end{align*}
    by requiring that
    \begin{align*}
        \exp(rX) n(z,\tau) a(t) \in NAMk(\alpha(r,X,z,\tau,t),\beta(r,X,z,\tau,t)),
    \end{align*}
   where $k(\alpha,\beta)$ is given by \eqref{eq: isom from SU(2) to K0/M}. We assume that $\delta_0$ is sufficiently small so that $k(\alpha,\beta)=\exp(r^\prime X^\prime)$ for some $0\leq r^\prime<\delta_0$ and $X^\prime\in U_{\fk/\fm}$.
   Applying \eqref{eq: exp ri tilde} and \eqref{eq: explicit formula for par A/par t} we obtain
   \begin{align*}
        1-\frac{\partial}{\partial t} A(\exp(rX) n(z,\tau) a(t)) = 1-\mathrm{Re}(\alpha)=1-\cos(r^\prime)=2\sin^2(r^\prime/2).
    \end{align*}
    By Lemma \ref{lem: change of variable}, if we write 
    \[
    \exp(rX)n(z,\tau)a(t)\in NAM\exp(r'X'),
    \]
    we can choose $0<\delta<\delta_0$ sufficiently small such that $\sin(r^\prime/2)$ vanishes for  $(r,X,z,\tau,t)\in(-\delta,\delta)\times U_{\fk/\fm}\times\cD$ if and only if $r=0$. Lemma \ref{lem: change of variable} implies that $\partial r^\prime / \partial r$ is analytic and never vanishes on $\{ 0\}\times U_{\fk/\fm}\times\cD$. It may be seen that there is a real analytic function $\xi_0$ on $(-2\delta,2\delta)\times U_{\fk/\fm} \times \cD$ such that $\sin(r^\prime/2)= r\xi_0(r,X,z,\tau,t)$. Defining $\xi = 2\xi_0^2$ gives the result. 
\end{proof}

\begin{lemma}\label{lem: uniform boundedness}
    Let $\delta > 0$ and $\cD\subset\BC\times\BR\times\BR$ be a bounded open set. There exists $\sigma>0$ depending on $\delta$ and $\cD$ such that
    \begin{align*}
        1-\frac{\partial}{\partial t} A(k n(z,\tau) a(t)) \geq \sigma,
    \end{align*}
    for arbitrary $k\in K_0$ with $d(k,M)\geq \delta$, and $(z,\tau,t)\in\cD$.
\end{lemma}
\begin{proof}
    Define the functions $\alpha,\beta: K_0\times\cD\to \BC$ by requiring that
    \begin{align*}
        kn(z,\tau) a(t) \in NAMk(\alpha(k,z,\tau,t),\beta(k,z,\tau,t)).
    \end{align*}
    If $d(k,M)\geq\delta$, then by Lemma \ref{lem: change of variable} $d(k(\alpha,\beta),e)\gg_\cD\delta$ for all $(z,\tau,t)\in\cD$. Hence, by \eqref{eq: isom from SU(2) to K0/M}, $\min\{|\alpha-1|^2,|\beta|^2\}\gg\delta^2$. Using $|\alpha|^2+|\beta|^2=1$ and $|\alpha-1|^2+|\beta|^2\gg\delta^2$, we obtain that $1-\mathrm{Re}(\alpha)\gg\delta^2$.
    For $(z,\tau,t)\in\cD$ and $d(k,M)\geq\delta$, \eqref{eq: explicit formula for par A/par t} implies that
    \begin{align*}
        1-\frac{\partial}{\partial t} A(k n(z,\tau) a(t)) = 1-\mathrm{Re}(\alpha(k,z,\tau,t))\gg\delta^2,
    \end{align*}
   which completes the proof.
\end{proof}

\subsection{Derivative with respect to \texorpdfstring{$K_0$}{K0}}

We shall also need the following estimate for the gradient with respect to $B=K_0/M$.
\begin{lemma}\label{lem: der wrt K0}
     Let $1/4 > \delta >0$ and $\cD\subset\BC\times\BR\times\BR$ be a bounded open set. There exists a small constant $C>0$ depending on $\cD$,  such that
     \begin{align*}
         &\left|\left( \frac{\partial}{\partial r_1},  \frac{\partial}{\partial r_2},\frac{\partial}{\partial r_3}\right)   A(\exp(r_1\tilde{X_1}+r_2\tilde{X_2}+r_3\tilde{X_3})n(z,\tau) a(t))   \right|^2\\
         :=&  \sum_{i=1}^3 \left(\frac{\partial}{\partial r_i}   A(\exp(r_1\tilde{X_1}+r_2\tilde{X_2}+r_3\tilde{X_3})n(z,\tau) a(t))   \right)^2
         \asymp  |z|^2+\tau^2,
     \end{align*}
     for $(z,\tau,t)\in\cD$ with $|z|^2+\tau^2\geq\delta^2$, and $r=(r_1^2+r_2^2+r_3^2)^{1/2}<C\delta$. Here $\tilde{X_1}, \tilde{X_2},\tilde{X_3}$ are defined as in \eqref{eq: basis for k/m}.
\end{lemma}
\begin{proof}
    We write $k(\alpha,\beta)=\exp(r_1\tilde{X_1}+r_2\tilde{X_2}+r_3\tilde{X_3})$ with $\alpha,\beta$ functions of $r_1,r_2,r_3$ given by \eqref{eq: exp ri tilde}. It may be seen that
    \begin{align}
        \alpha=1+ir_3 + r^2 A(r_1,r_2,r_3), \label{eq: taylor for alpha}\\
        \beta=ir_1-r_2+ r^2 B(r_1,r_2,r_3).\label{eq: taylor for beta}
    \end{align}
    where $A(r_1,r_2,r_3)$ and $B(r_1,r_2,r_3)$ are bounded and have bounded derivatives when $r\ll 1$.
    By Lemma \ref{lem: explicit formula for A},
    \[
    A(\exp(r_1\tilde{X_1}+r_2\tilde{X_2}+r_3\tilde{X_3})n(z,\tau) a(t))=t-\log\left(\left| \frac{\alpha-1}{2}(-e^t-|z|^2+i\tau)-\beta \overline{z}+\frac{\alpha+1}{2} \right|^{2}\right).
    \]
    Applying \eqref{eq: taylor for alpha} and \eqref{eq: taylor for beta} gives that 
    \begin{align*}
        \left(\frac{\partial}{\partial r_1},\frac{\partial}{\partial r_2},\frac{\partial}{\partial r_3}\right)A(\exp(r_1\tilde{X_1}+r_2\tilde{X_2}+r_3\tilde{X_3})n(z,\tau) a(t))=(2\mathrm{Im}(z),-2\mathrm{Re}(z),\tau) + O(r),
    \end{align*}
    whose squared magnitude is $4|z|^2+\tau^2+O(r)$.
    Then the result follows.
\end{proof}

\section{Bounds near the spectrum}\label{sec: bd near the spectrum}
In this section, we prove the bound \eqref{eq: bd for I(g)} in Proposition \ref{prop: bound for I}, by building up the integral $I(\lambda,\phi,g)$ in several steps.
\subsection{Oscillatory integrals}
We first estimate two one-dimensional integrals that appear in $I(\lambda,\phi,g)$.

\begin{proposition}\label{prop: first one dimensional intgeral estimate}
    Let $B,C>0$ and $\epsilon_0,\epsilon'>0$ be constants.
    Let $\chi_0\in C_c^\infty (\BR)$ be a smooth function supported in $[-1,1]$. 
    If $(z,\tau,t)\in\BC\times\BR\times\BR$ satisfying $|z|,|\tau|,|t|<C$, $k\in K_0$, and $s^\prime>0$ satisfying
    \begin{align}\label{eq: cond for k,s first int}
        d(k,M)\geq Bs^{-1/2 + \epsilon_0}\beta^{1/2}\quad\text{ and }\quad       |s^\prime-s|\leq 3\beta
    \end{align}
    for some $s,\beta\geq 1$ with $\beta\leq 3s^{1-\epsilon'}$, then
    \begin{align}\label{eq: first int where k big}
        \int_{-\infty}^\infty \chi_0(t^\prime) \exp\left( is^\prime t^\prime - is A(kn(z,\tau)a(t+t^\prime))  \right) dt^\prime\ll s^{-A}.
    \end{align}
    The implied constant depends on $A,B,C,\epsilon_0,\epsilon'$ and the size of the first $n$ derivatives of $\chi_0$, where $n$ depends on $\epsilon_0,\epsilon'$ and $A$.
\end{proposition}
\begin{proof}
    We denote by $\cD\subset\BC\times\BR\times\BR$ the set consisting of triples $(z,\tau,t)$ satisfying  $|z|,|\tau|,|t|<C+2$. By applying Lemma \ref{lem: uniformazation lemma}, we see that there is a $\delta>0$ and a nonvanishing real analytic function  $\xi$ on $(-\delta,\delta)\times U_{\fk/\fm} \times \cD$ such that
    \begin{align*}
        \frac{\partial}{\partial t} A(\exp(rX)n(z,\tau)a(t)) = 1-r^2\xi(r,X,z,\tau,t),
    \end{align*}
    when $r\in (-\delta,\delta)$, $X\in U_{\fk/\fm}$, $(z,\tau,t)\in \cD$. If $Z(r,X,z,\tau,t)$ is an antiderivative of $\xi$ with respect to $t$ that is smooth as a function of $(r,X,z,\tau,t)$, we may integrate this to obtain
    \begin{align*}
        A(\exp(rX)n(z,\tau)a(t)) = t - r^2 Z(r,X,z,\tau,t) + c(r,X,z,\tau)
    \end{align*}
    for some function $c(r,X,z,\tau)$.
    
    If $k=\exp(rX)$ for some $r\in (-\delta,\delta)$ and $X\in U_{\fk/\fm}$, we may use this to rewrite the integral \eqref{eq: first int where k big} as
    \begin{align}\label{eq: rewrite 1st int where k big}
         &\int_{-\infty}^\infty \chi_0(t^\prime) \exp\left( is^\prime t^\prime - is A(\exp(rX)n(z,\tau)a(t+t^\prime))  \right) dt^\prime\notag\\
         =&\int_{-\infty}^\infty \chi_0(t^\prime) \exp\left( is^\prime  t^\prime  - is (t+t^\prime) + isr^2 Z(r,X,z,\tau,t+t^\prime) - is c(r,X,z,\tau)  \right) dt^\prime \notag\\
         =& e^{-is(t+c(r,X,z,\tau))} \int_{-\infty}^\infty \chi_0(t^\prime) \exp\left( i(s^\prime-s)t^\prime + isr^2 Z(r,X,z,\tau,t+t^\prime)\right) dt^\prime  \notag \\
         =& e^{-is(t+c(r,X,z,\tau))} \int_{-\infty}^\infty \chi_0(t^\prime) \exp\left( isr^2 \Psi(t^\prime)\right) dt^\prime,
    \end{align}
    where we define
    \begin{align*}
        \Psi(t^\prime) = Z(r,X,z,\tau,t+t^\prime) + s^{-1}r^{-2}(s^\prime-s)t^\prime .
    \end{align*}

    Our assumption \eqref{eq: cond for k,s first int} implies that
    $r \geq B_1 s^{-1/2 + \epsilon_0}\beta^{1/2}$ for some constant $B_1 > 0$, and
    \begin{align*}
        |s^{-1}r^{-2}(s^\prime-s)| \ll s^{-1} \left( B_1 s^{-1/2 + \epsilon_0}\beta^{1/2} \right)^{-2} \beta \ll s^{-2\epsilon_0}.
    \end{align*}
    Hence,
    \begin{align}\label{eq: est for Psi}
        \begin{split}
            \Psi &= Z(r,X,z,\tau,t+t^\prime) + O\left(s^{-2\epsilon_0}\right)t^\prime, \quad\text{ and }\\
            \frac{\partial\Psi}{\partial t^\prime}&= \xi(r,X,z,\tau,t+t^\prime)  + O\left(s^{-\epsilon_0}\right).
        \end{split}
    \end{align}
    It follows from \eqref{eq: xi estimate} and \eqref{eq: est for Psi} that there exists a $\sigma>0$ so that, for $s$ sufficiently large, $|\partial\Psi/\partial t^\prime| > \sigma/2$ for all $r\in(-\delta,\delta)$, $|z|,|\tau|,|t|<C$ and $t^\prime\in(-1,-1)$. In addition, all other derivatives of $\Psi$ are bounded from above. As $r \geq B_1 s^{-1/2 + \epsilon_0}\beta^{1/2}$ implies that $sr^2 \geq  B_1^2 s^{2\epsilon_0}\beta \geq B_1^2 s^{2\epsilon_0}$, the bound \eqref{eq: first int where k big} follows by integration by parts in \eqref{eq: rewrite 1st int where k big}.

    In the case where $k$ can not be written as $\exp(rX)$ for $r\in (-\delta,\delta)$ and $X\in U_{\fk/\fm}$, we have $d(k,M)\gg\delta$, so Lemma \ref{lem: uniform boundedness} implies that
    \begin{align*}
        1-\frac{\partial}{\partial t^\prime} A(k n(z,\tau) a(t+t^\prime)) \geq C_1
    \end{align*}
    for some $C_1 > 0$. It gives
    \begin{multline*}
        \left|\frac{\partial}{\partial t^\prime} \left(s^\prime s^{-1} t^\prime - A(kn(z,\tau)a(t+t^\prime)) \right)\right| \geq \left(  1-\frac{\partial}{\partial t^\prime} A(k n(z,\tau) a(t+t^\prime)) \right)- \left|s^\prime s^{-1} - 1 \right|\\ 
        \geq  \,C_1 - 3\beta s^{-1}  =C_1 + O(s^{-\epsilon'})\gg1.
    \end{multline*}
    The result also follows by the integration by parts.
\end{proof}

The second one-dimensional integral that we shall estimate is as follows.

\begin{proposition}\label{prop: second one dimensional integral estimate}
    Let $B,C>0$ and $\epsilon_0,\epsilon'>0$ be constants. Let $\chi_0\in C_c^\infty(\BR)$ be a smooth function supported in $[-1,1]$. If $(z,\tau,t)\in\BC\times\BR\times\BR$ and $s^\prime > 0$ satisfying
    \begin{align*}
        |z|,|\tau|,|t|<C,\qquad |z|^2+|\tau|^2\geq \left(B s^{-1/2 + \epsilon_0} \beta^{1/2}\right)^2\qquad\text{ and }\quad|s^\prime-s|\leq 3\beta
    \end{align*}
    for some $s,\beta\geq 1$ with $\beta\leq 3s^{1-\epsilon'}$,
    then
    \begin{align}\label{eq: second int est}
        \int_{-\infty}^\infty \chi_0(t^\prime) e^{is^\prime t^\prime} \varphi_s\left(n(z,\tau)a(t+t^\prime) \right) dt^\prime       \ll s^{-A}.
    \end{align}
    The implied constant depends on $A,B,C,\epsilon_0,\epsilon'$ and the size of the first $n$ derivatives of $\chi_0$, where $n$ depends on $\epsilon_0,\epsilon'$ and $A$.
\end{proposition}
\begin{proof}
    If we apply the functional equation $\varphi_s = \varphi_{-s}$, and substitute $\varphi_{-s}$ by the formula \eqref{eq:HC K int formula} into the left-hand side of \eqref{eq: second int est}, it becomes
    \begin{align*}
        \int_{-\infty}^\infty \int_{M\backslash K_0} \chi_0(t^\prime) \exp\left( is^\prime t^\prime  +(1-is) A(kn(z,\tau)a(t+t^\prime))\right) d\overline{k} dt^\prime .
    \end{align*}
   Let $C_1>0$ be a constant to be chosen later. Let $\chi_1,\chi_2\in C_c^\infty(M\backslash K_0)$ be such that $0\leq \chi_1,\chi_2\leq 1$, $\chi_1 + \chi_2 \equiv 1$, and
    \begin{align*}
        &\supp(\chi_1)\subset \left\{ Mk \in M\backslash K_0 \mid k = \exp(rX), r\in [-2C_1 s^{-1/2+\epsilon_0}\beta^{1/2},2C_1 s^{-1/2+\epsilon_0}\beta^{1/2}], X\in U_{\fk/\fm} \right\},           \\
        &\supp(\chi_2)\subset M\backslash K_0 - \left\{ Mk \in M\backslash K_0 \mid  k = \exp(rX), r\in [-C_1 s^{-1/2+\epsilon_0}\beta^{1/2},C_1 s^{-1/2+\epsilon_0}\beta^{1/2}], X\in U_{\fk/\fm} \right\}.
    \end{align*}
    Proposition \ref{prop: first one dimensional intgeral estimate} implies that
    \begin{align*}
        \int_{-\infty}^\infty \int_{M\backslash K_0} \chi_2(\overline{k}) \chi_0(t^\prime) \exp\left( is^\prime t^\prime  +(1-is) A(kn(z,\tau)a(t+t^\prime))\right) d\overline{k} dt^\prime \ll_A s^{-A}.
    \end{align*}
    Hence, it suffices to estimate
    \begin{align*}
         &\int_{-\infty}^\infty \int_{M\backslash K_0} \chi_1(\overline{k}) \chi_0(t^\prime) \exp\left( is^\prime t^\prime   +(1-is) A(kn(z,\tau)a(t+t^\prime))\right) d\overline{k} dt^\prime\\
         =&\int_{-\infty}^\infty\left( \int_{M\backslash K_0} \chi_1(\overline{k})  \exp\left(   (1-is) A(kn(z,\tau)a(t+t^\prime))\right) d\overline{k} \right)\chi_0(t^\prime)e^{is^\prime t^\prime } dt^\prime.
    \end{align*}
    We shall do this by showing the inner integral above is $O(s^{-A})$ as well. We use the local chart given by \eqref{eq: basis for k/m} to rewrite the inner integral as
    \begin{align}\label{eq: oscillatory int on K}
         \iiint_{-\infty}^\infty  \tilde{\chi_1}(r_1,r_2,r_3)\exp\left(   -is A(\exp(r_1\tilde{X_1} + r_2\tilde{X_2}+r_3\tilde{X_3})n(z,\tau)a(t))\right) dr_1 dr_2 dr_3.
    \end{align}
    Here $|z|,|\tau|<C$ satisfy $|z|^2+|\tau|^2 \geq (Bs^{-1/2+\epsilon_0}\beta^{1/2})^2$ and $t\in (-C-1,C+1)$, and $\tilde{\chi_1}(r_1,r_2,r_3) = \tilde{\chi_1}(r_1,r_2,r_3;z,t)$ is the smooth compactly supported function obtained by combining all of the amplitude factors. Moreover,
    \begin{align*}
        \supp(\tilde{\chi_1}) \subset \left\{ (r_1,r_2,r_3)\in\BR^3 \mid r =(r_1^2 +  r_2^2+r_3^2)^{1/2} \in [-2C_1 s^{-1/2+\epsilon_0}\beta^{1/2},{2}C_1 s^{-1/2+\epsilon_0}\beta^{1/2}]  \right\}.
    \end{align*}
    
    We let $\delta=Bs^{-1/2+\epsilon_0}\beta^{1/2}$ and $\phi(r_1,r_2,r_3)=A(\exp(r_1\tilde{X_1} + r_2\tilde{X_2}+r_3\tilde{X_3})n(z,\tau)a(t))$.
Lemma \ref{lem: der wrt K0}  implies that if $C_1$ is chosen to be sufficiently small then
    \begin{align*}
         \left|\nabla  \phi(r_1,r_2,r_3)  \right|^2 \asymp  |z|^2+|\tau|^2 \ge \delta^2,
     \end{align*}
     when $|z|,|\tau|,|t|<C+1$, $|z|^2+|\tau|^2\ge\delta^2$ and $(r_1,r_2,r_3) \in \supp(\tilde{\chi_1})$. Moreover, \eqref{eq: oscillatory int on K} equals
    \begin{multline*}
        \int  \tilde{\chi_1}(r_1,r_2,r_3)e^{  -is \phi(r_1,r_2,r_3)} dr_1 dr_2 dr_3 
        = \delta^3 \int \tilde{\chi_1}(\delta(r_1,r_2,r_3)) e^{-i(s\delta^2)(\delta^{-2}\phi(\delta (r_1,r_2,r_3))} dr_1dr_2dr_3.
    \end{multline*}
    The function $\tilde{\chi_1}(\delta(r_1,r_2,r_3))$ is now a cutoff function at scale $1$, and $(\delta)^{-2}\phi(\delta( r_1,r_2,r_3))$ satisfies
    \(     \nabla\left( \delta^{-2}\phi(\delta(r_1,r_2,r_3)) \right)\gg 1\) and all of its higher derivatives are bounded
    for $(r_1,r_2,r_3)\in \supp(\tilde{\chi_1}(\delta\cdot))$. Note that $s\delta^2\gg s^{2\epsilon_0}\beta$, so the desired bound for \eqref{eq: oscillatory int on K} follows by integration by parts.
\end{proof}

\subsection{Proof of \eqref{eq: bd for I(g)} in Proposition \ref{prop: bound for I}}
Let $\chi_0 \in C_c^\infty(\BR)$ be a smooth function supported in $[-1,1]$, $s_1,s_2 \in \BR$ and $g\in G_0$. We define
\begin{align}\label{eq: integral J(s t1 t2 g)}
    J(s,s_1,s_2,g;\chi_0)= \iint_{-\infty}^\infty \chi_0(t_1) \chi_0(t_2) e^{-is_1 t_1 + is_2 t_2} \varphi_s \left( a(-t_1)ga(t_2) \right) dt_1 dt_2.
\end{align}
We now combine the one-dimensional results Proposition \ref{prop: first one dimensional intgeral estimate} and \ref{prop: second one dimensional integral estimate} to bound $J(s,s_1,s_2,g;\chi_0)$.

\begin{proposition}\label{prop: nonstationary oscillatory integral estimate}
    Let $B_0,C_0 >0$ and $1/8>\epsilon_0,\epsilon'>0$ be given. If $g\in G_0$ satisfies
    \begin{align*}
        d(g,e)\leq C_0\quad\text{ and }\quad d(g,MA)\geq B_0 s^{-1/2 + \epsilon_0}\beta^{1/2}
    \end{align*}
    for some $s,\beta\geq 1$ satisfying $\beta \leq 3 s^{1-\epsilon'}$, and $s_1,s_2\in [s-3\beta,s+3\beta]$, then
    \begin{align*}
        J(s,s_1,s_2,g;\chi_0)\ll s^{-A}.
    \end{align*}
    The implied constant depends on $A,B_0,C_0, \epsilon_0,\epsilon'$ and the size of the first $n$ derivatives of $\chi_0$, where $n$ depends on $\epsilon_0,\epsilon'$ and $A$.
\end{proposition}
\begin{proof}
    We define the map $u: K_0\times \BR\to K_0$ by requiring $ka(-t_1)\in NA u(k,t_1)$. By Lemma \ref{lem: change of variable}, for each fixed $t_1\in\BR$,  $u(\cdot,t_1)$ is a diffeomorphism from $K_0$ onto itself and induces a diffeomorphism from $M\backslash K_0$ onto itself. If we apply the functional equation $\varphi_s = \varphi_{-s}$, write $\varphi_{-s}$ as an integral over $K_0$ by \eqref{eq:HC K int formula}, and apply Lemma \ref{lem: spliting A}, we obtain
    \begin{align*}
        \varphi_{s}\left(  a(-t_1)ga(t_2)   \right)          =& \int_{M\backslash K_0} \exp\big( (1-is)   \left(A\left( u(\overline{k},t_1)ga(t_2) \right) - A\left( u(\overline{k},t_1)a(t_1)\right) \right)\big) d\overline{k}\\
        =&\int_{M\backslash K_0} \exp\big( (1-is)   \left(A\left( \overline{u}ga(t_2) \right) - A\left( \overline{u}a(t_1)\right) \right)\big) \left|\det \frac{\partial \overline{k}}{\partial \overline{u}}\right| d\overline{u}.
    \end{align*}
    Here we use $\partial \overline{u}/\partial\overline{k}$ to denote the Jacobian matrix of the diffeomorphism $u(\cdot,t_1):M\backslash K_0\to M\backslash K_0$, and $\partial \overline{k}/\partial \overline{u}$ is the inverse matrix. Substituting this into the definition \eqref{eq: integral J(s t1 t2 g)} of $J(s,s_1,s_2,g;\chi_0)$ gives
    \begin{align}\label{eq: expanding J}
    \begin{split}
         \iint_{-\infty}^\infty &\int_{M\backslash K_0}\chi_0(t_1) \chi_0(t_2) e^{-is_1 t_1 +is_2 t_2}\\
         &\exp\big( (1-is)   \left(A\left( \overline{u}ga(t_2) \right) - A\left( \overline{u}a(t_1)\right) \right)\big) \left|\det \frac{\partial \overline{k}}{\partial \overline{u}}\right| d\overline{u} dt_1 dt_2 . 
         \end{split}
    \end{align}
    Let $g = k^\prime n(z^\prime,\tau^\prime) a (t^\prime)$ with $k^\prime \in K_0$, $z^\prime\in\BC$ and $\tau^\prime, t^\prime\in\BR$. The condition $d(g,e)\leq C_0$ implies that $z^\prime$, $\tau$ and $t^\prime$ are bounded. Choose a constant $C>0$.  If $d(\overline u,M) \geq Cs^{-1/2 + \epsilon_0}\beta^{1/2}$, then integrating \eqref{eq: expanding J} in $t_1$ and applying Proposition \ref{prop: first one dimensional intgeral estimate} shows that the integral of \eqref{eq: expanding J} over $t_1$ and $t_2$ with this $\overline u$ is $O_A(s^{-A})$. And if $d(\overline u,M) < Cs^{-1/2 + \epsilon_0}\beta^{1/2}$ but $d(k^\prime,M) \geq C_1 C s^{-1/2 + \epsilon_0}\beta^{1/2}$ for some constant $C_1>0$ independent of the choice of $C$ such that $d(\overline uk^\prime,M) \geq Cs^{-1/2 + \epsilon_0}\beta^{1/2}$, then we obtain the same conclusion by integrating in $t_2$. Combining these, we see that \eqref{eq: expanding J} will be $O_A( s^{-A})$ unless $d(k^\prime,M) < C_1 C s^{-1/2 + \epsilon_0}\beta^{1/2}$, and we assume that this is the case.

    If $C$ is chosen sufficiently small, the condition $d(k^\prime,M) < C_1 C s^{-1/2 + \epsilon_0}\beta^{1/2}$ and our assumption that $d(g,MA)\geq B_0 s^{-1/2 + \epsilon_0}\beta^{1/2}$ imply that $|z^\prime|^2+|\tau^\prime|^2 \geq (C_2 s^{-1/2 + \epsilon_0}\beta^{1/2})^2 $ for some $C_2>0$ depending only on $B_0$. For $-1\leq t_1 \leq 1$, we define $z^\prime(t_1)\in\BC$ and $\tau^\prime(t_1), t^\prime(t_1) \in \BR$ by requiring
    \begin{align*}
        a(-t_1)g = a(-t_1)k^\prime n(z^\prime,\tau^\prime) a (t^\prime) \in  K_0n(z^\prime(t_1),\tau^\prime(t_1)) a(t^\prime(t_1)).
    \end{align*}
    It follows that if $C$ is sufficiently small, 
    we have  $|z^\prime(t_1)|^2+|\tau^\prime(t_1)|^2 \geq (C_3 s^{-1/2 + \epsilon_0}\beta^{1/2})^2 $ for some  $C_3>0$ and for all $-1\leq t_1 \leq 1$. The result now follows by applying Proposition \ref{prop: second one dimensional integral estimate} to the integral \eqref{eq: integral J(s t1 t2 g)} for each fixed $t_1$.
\end{proof}

\begin{proof}[Proof of \eqref{eq: bd for I(g)} in Proposition \ref{prop: bound for I}]
    We can unfold the integral $I(\lambda,\phi,g)$ by the Fourier inversion for $\phi$ and inverse Harish-Chandra transform for $h_\lambda$, that is, $I(\lambda,\phi,g)$ is equal to
    \begin{align*}
        \frac{1}{(2\pi)^2}\int_0^\infty\Big( &\iint_{\BC\times\BR\times\BR}  \overline{\sF_t\phi(z_1,\tau_1,s_1})\sF_t\phi(z_2,\tau_2,s_2)\\
        &J(s,s_1,s_2,n(z_1,\tau_1)^{-1}gn(z_2,\tau_2);\chi) dz_1 dz_2 d\tau_1 d\tau_2 ds_1 ds_2\Big)h_\lambda(s) d\nu(s).
    \end{align*}
    The integral with $s_1,s_2 \in [\lambda-\beta,\lambda+\beta]$, and $\chi\in C_c^\infty(\overline{\BD} \times \overline{\BD} \times \BR \times \BR)$ defined by
    \begin{align*}
        \chi(z_1,\tau_1,z_2,\tau_2,t_1,t_2) = \overline{b(z_1,\tau_1,t_1)}b(z_2,\tau_2,t_2)16e^{-2t_1 -2 t_2} .
    \end{align*}
    By the rapid decay of $h_\lambda$ away from $\lambda$, as $\beta\geq\lambda^{\epsilon'}$, we may also assume $s\in [\lambda-\beta,\lambda+\beta]$.
    From the facts that $|z_1|,|\tau_1|,|z_2|,|\tau_2| \leq \lambda^{-1/2}$ and $d(g,MA)\geq\lambda^{-1/2+\epsilon_0}\beta^{1/2}$, it may be seen that
    \begin{align*}
        d(n(z_1,\tau_1)^{-1}gn(z_2,\tau_2), MA) \gg \lambda^{-1/2 + \epsilon_0}\beta^{1/2}.
    \end{align*}
    The result now follows from Proposition \ref{prop: nonstationary oscillatory integral estimate}.
\end{proof}

\section{Bounds away from the spectrum}\label{sec: bd away from spec}

In this section, we prove the bound \eqref{eq: bd for I(e)} in Proposition \ref{prop: bound for I}. 
We first decompose $k_\lambda$ along the radial direction. Without loss of generality, we assume the support of $k_\lambda(a(t))$ is contained in $[-1,1]$. Let $b_0\in C_c^\infty(\BR)$ be an even cutoff function that is equal to $1$ on $[-1,1]$ and zero outside $[-2,2]$. Moreover, we suppose $b_0$ is even and $0\leq b_0\leq 1$. Then $k_\lambda(a(t)) = b_0(t) k_\lambda(a(t))$ for any $t\in\BR$. Let $\epsilon_0 > 0$. Define $b_1 \in C_c^\infty(\BR)$ to be the function
\[
b_1(x) = b_0 (\beta^{1/2-\epsilon}x)
\]
so that $b_1$ satisfies:
\begin{itemize}
    \item $0\leq b_1 \leq 1$ and $b_1$ is even,
    \item $b_1(t) = 1$ for $|t|\leq \beta^{-1/2 + \epsilon_0}$,
    \item $b_1(t) = 0$ for $|t| \geq 2\beta^{-1/2 + \epsilon_0}$.
\end{itemize}
Let $b_2 = b_0 - b_1\in C_c^{\infty}(\BR)$.
Since $b_i$ are constant in some neighborhood of $0$, we can extend the domains of $b_i$ to $\BH^2_\BC$ by the polar coordinate map, i.e.,
\begin{align}\label{eq:def of bi}
    b_i(n(z,\tau)a(t)):=b_i(t_0),\text{ where }t_0\text{ satisfies the relation }n(z,\tau)a(t)\in K_0a(t_0)K_0,
\end{align}
so that $b_i \in C_c^\infty(\BH^2_\BC)$ are $K_0$-bi-invariant, and the restrictions $b_i\circ a(t) $ are the original functions, $i=0,1,2$. Let $k_1, k_2\in C_c^\infty(\BH^2_\BC)$ be $K_0$-bi-invariant functions so that $k_i = b_i k_\lambda$. Then we decompose the integral $I(\lambda,\phi,e) = I_1(\phi) + I_2(\phi)$ where
\begin{align*}
    I_i(\phi)=\iint_{\BC\times\BR\times \BR} \overline{b\phi(z_1,\tau_1,t_1})b\phi(z_2,\tau_2,t_2) &k_i ((n(z_1,\tau_1)a(t_1))^{-1}n(z_2,\tau_2)a(t_2)) \\
    &16e^{-2t_1 -2 t_2} dz_1d\tau_1dt_1 dz_2d\tau_2dt_2,\qquad i=1,2.
\end{align*}
By the choice of the width of $\supp(b_1)$,
we will show that for $\phi \in \cS(\overline{\BD}\times\BR) $  satisfying $\|\phi \|_2 = 1$ and $\supp(\sF_t\phi) \subset \overline{\BD}\times (\BR\backslash \pm[\lambda-\beta,    \lambda+\beta])$,
\begin{align}
   & I_1(\phi) \ll \beta^{-1/2 + \epsilon_0} ,\label{eq: bound of I1}\\
    &I_2(\phi)\ll_{\epsilon_0,A}  \lambda^{-A} .\label{eq: bound of I2}
\end{align}
Hence, the bound \eqref{eq: bd for I(e)} in Proposition \ref{prop: bound for I} will follow from \eqref{eq: bound of I1} and \eqref{eq: bound of I2}. We will prove \eqref{eq: bound of I1} and \eqref{eq: bound of I2} in the rest of the paper.

\subsection{Bound for $I_1(\phi)$}
We define $\Phi \in C_c^\infty(\BH^2_\BC)$ by $\Phi(n(z,\tau)a(t)\cdot o):= b\phi(z,\tau,t)$. It may be seen that $I_1(\phi) = \langle \Phi \times k_1, \Phi \rangle_{\BH^2_\BC}$. Here $\langle \cdot, \cdot  \rangle_{\BH^2_\BC}$ is the inner product on $L^2(\BH^2_\BC)$ with respect to the volume form. By the Plancherel formula for the Helgason transform \eqref{eq: PL formula}, 
\begin{align*}
    I_1 (\phi) = \int_0^\infty\int_{M\backslash K_0} \widehat{\Phi}(s,\overline{k})\widehat{k_1}(s)\overline{\widehat{\Phi} (s,\overline{k})} d\nu(s) d\overline{k}\ll \| \widehat{k_1} \|_\infty \| \Phi\|_2^2 \ll \| \widehat{k_1} \|_\infty.
\end{align*}
We will prove the following pointwise bound for $\widehat{k_1}$, and the bound \eqref{eq: bound of I1} will follow from this immediately.
\begin{proposition}
\begin{itemize}
    \item If $|s|\leq \lambda/2$, then $\widehat{k_1}(s) \ll_A \lambda^{-A}$.
    \item If $|s|\geq \lambda/2$, then $\widehat{k_1}(s) \ll \beta^{-1/2 + \epsilon_0}$.
\end{itemize}
\end{proposition}

We first prove the case $|s| \geq \lambda/2$. We write the integral via the polar coordinate (see e.g. \cite[Ch. I, Theorem 5.8]{helgason1984groups}). There exists a constant $c > 0$ so that
\begin{align*}
    \widehat{k_1}(s)  &= \int_{\BH_\BC^2}  b_1k_\lambda(x) \varphi_{-s}(x) d\mathrm{Vol}(x)\\
    &= c \int_0^\infty b_1(t) k_\lambda(a(t)) \varphi_{-s}(a(t)) \sinh^2(t/2)\sinh(t) dt.
\end{align*}
The bounds $k_\lambda(a(t)) \ll \lambda^{3} (1+\lambda|t|)^{-3/2}$ from \cite[Lemma 2.8]{marshall2016p}, and $\varphi_{-s}(a(t)) \ll (1 + |st|)^{-3/2}$ from \cite[Theorem 1.3]{marshall2016p} with $|s| \geq \lambda/2$ imply that
\begin{align*}
    \widehat{k_1}(s) \ll \int_0^\infty b_1(t) (\lambda^{3/2} t^{-3/2})(\lambda t)^{-3/2} \sinh^2(t/2)\sinh(t) dt \ll \beta^{-1/2 + \epsilon_0}.
\end{align*}

Now we consider the case $|s| \leq \lambda/2$. We apply the inverse Harish-Chandra transform to $k_\lambda$, and apply \eqref{eq:HC K int formula}. We obtain
\begin{align*}
    \widehat{k_1}(s) &= \int_0^\infty \left(\int_{\BH^2_\BC}  b_1(x) \varphi_r(x) \varphi_{-s}(x) d\mathrm{Vol}(x) \right)h_\lambda(r)  d\nu(r) \\
    &=\int_0^\infty \left(\int_{\BH^2_\BC} \int_{K_0} b_1(x) \varphi_r(x) \exp((1-is)A(kx)) dk d\mathrm{Vol}(x)\right)h_\lambda(r)  d\nu(r)\\
    &=\int_0^\infty \left(\int_{\BH^2_\BC}b_1(x) \varphi_r(x) \exp((1-is)A(x))d\mathrm{Vol}(x) \right)h_\lambda(r)  d\nu(r) \\
    &= \int_0^\infty \int_{K_0} \left(\int_{\BH^2_\BC} b_1(x) \exp((1+ir)A(kx)) \exp((1-is)A(x)) d\mathrm{Vol}(x) \right)h_\lambda(r) dk d\nu(r).
\end{align*}
It suffices to show for $|r-\lambda|\le\beta/4$,
\begin{align*}
    \int_{\BH^2_\BC} b_1(x) \exp((1+ir)A(kx)) \exp((1-is)A(x)) d\mathrm{Vol}(x)\ll_A r^{-A}.
\end{align*}
Combining all amplitude factors and applying the Iwasawa coordinates \eqref{eq: Iwa coord action} for $\BH^2_\BC$, it suffices to show
\begin{align}\label{eq: bd for J(k)}
    J(k) :=\int_{\BR^4} \chi_1(x,y,\tau,t) \exp(ir(A(kn(x+iy,\tau)a(t))-(s/r)t)) dx dy d\tau dt \ll_A r^{-A}.
\end{align}
Here $\chi_1(z,y,\tau,t) \in C_c^\infty(\BR^4)$ is a cutoff function at scale $\beta^{-1/2 + \epsilon_0}$, and the $k$-variable is ignored for notation simplicity. We denote by $\rho = s/r$, which satisfies $|\rho| \leq 2/3$.  The phase function with parameters $k\in K_0$ and $\rho$ is
\begin{align*}
    \phi(x,y,\tau,t;k,\rho) = A(kn(x+iy,\tau)a(t))-\rho t.
\end{align*}

\begin{lemma}
    If $\chi \in C_c^\infty(\BR^4)$ is a cutoff function at scale $\beta^{-1/2 + \epsilon_0}$, and $\phi \in C^\infty(\BR^4)$ is real-valued and satisfies
    \begin{align*}
        |\nabla\phi(x)| \geq c > 0, \;\; \phi^{(n)}(x)\ll_n 1
    \end{align*}
    for some fixed constant $c>0$ and for any $x\in \supp(\chi)$ and $n\geq 1$, then
    \begin{align*}
        \int  \chi(x) e^{ir\phi(x)} dx \ll_A r^{-A}.
    \end{align*}
\end{lemma}
\begin{proof}
    We have
    \begin{align*}
        \int \chi(x) e^{ir\phi(x)} dx = \beta^{-2 + 4\epsilon_0} \int \chi(\beta^{-1/2 + \epsilon_0} x) e^{i(r\beta^{-1/2 + \epsilon_0})(\beta^{1/2-\epsilon_0}\phi(\beta^{-1/2 + \epsilon_0} x))} dx.
    \end{align*}
    The cutoff function $\chi(\beta^{-1/2 + \epsilon_0} x)$ is now  at scale $1$.
    The lemma follows by integration by parts.
\end{proof}

Hence, to apply the above lemma to $J(k)$, it suffices to give a nonzero lower bound for the gradient 
\(
   \nabla \phi = (\partial_x\phi,\partial_y\phi,\partial_\tau\phi,\partial_t\phi) 
\)
when $(x,y,\tau,t)\in \supp(\chi_1)$. By the compactness of $\supp(\chi_1)$ and $K_0$, it suffices to show the gradient is nonvanishing. Let $\alpha = \alpha(x,y,\tau,t;k) $ and $\beta = \beta(x,y,\tau,t;k) $ be complex-valued functions satisfying $|\alpha|^2+|\beta|^2=1$ such that
\begin{align*}
    kn(x+iy,\tau)a(t)\in NAMk(\alpha,\beta).
\end{align*}
Here $k(\alpha,\beta)\in K_0$ is the same as in \eqref{eq: isom from SU(2) to K0/M}. We suppose that $\nabla\phi=0$. Then by applying \eqref{eq: explicit formula for par A/par t}--\eqref{eq: explicit formula for par A/par tau}, we see that $\mathrm{Re}(\alpha)=\rho$ and $\mathrm{Im}(\alpha)=\mathrm{Re}(\beta)=\mathrm{Im}(\beta)=0$, which contradicts to $|\alpha|^2+|\beta|^2=1$.

\subsection{Bound for $I_2(\phi)$}

By applying the inverse Harish-Chandra transform to $h_\lambda$, it suffices to prove the bound
\begin{align}\label{eq:equa bd for J2}
    J_2(\phi,s)\ll_{A,\epsilon_0} s^{-A}
\end{align}
for $|s-\lambda|\le\beta/2$, where
\begin{align*}
    J_2(\phi,s)=\iint_{\BC\times\BR\times\BR} \overline{b\phi(z_1,\tau_1,t_1)}b\phi(z_2,\tau_2,t_2)(b_2\varphi_s)\left((n(z_1,\tau_1)a(t_1))^{-1}n(z_2,\tau_2)a(t_2)\right)\\
    (16e^{-2t_1 -2 t_2})  dz_1 dz_2 d\tau_1d\tau_2 dt_1 dt_2.
\end{align*}
Note that if we let $z=e^{-t_1/2}(z_2-z_1)$ and $\tau=e^{-t_1}(\tau_2-\tau_1+2\Im(z_1\overline{z_2}))$ then
\begin{align}\label{eq: defn for z and tau}
    (n(z_1,\tau_1)a(t_1))^{-1}n(z_2,\tau_2)a(t_2) = n(z,\tau) a(t_2-t_1).
\end{align}
Therefore
\begin{align*}
    J_2(\phi,s)=\iint_{\BC\times\BR\times\BR} \overline{b\phi(z_1,\tau_1,t_1)}b\phi(z_2,\tau_2,t_2)(b_2\varphi_s)\left(\cA(z,\tau,t_2-t_1)\right)(16e^{-2t_1 -2 t_2})  dz_1 dz_2 d\tau_1d\tau_2 dt_1 dt_2.
\end{align*}
Here $\cA$ is the distance function satisfying \eqref{eq: dist function}.

\subsubsection{Derivatives of the distance function $\cA$}\label{subsubsection: Derivatives of the distance function}

\begin{lemma}\label{KN lower bound for t from lower bound for A(z,t)}
    Suppose $|z|\ll \lambda^{-1/2}$, $|\tau|\ll \lambda^{-1/2}$, and $\cA(z,\tau,t)\gg \beta^{-1/2 + \epsilon_0}$.
    Then we have $|t|\gg \beta^{-1/2 + \epsilon_0}$.
\end{lemma}
\begin{proof}
    This follows directly from the relation \eqref{eq: dist function}.
\end{proof}

\begin{proposition}
     Suppose $|z|\ll \lambda^{-1/2}$, $|\tau|\ll \lambda^{-1/2}$, $t\ll1$, and $\cA(z,\tau,t)\gg \beta^{-1/2 + \epsilon_0}$. Then we have
     \begin{align}
         &\frac{\partial}{\partial t}\cA(z,\tau,t)=\sgn t+O(\lambda^{-1}\beta^{1-2\epsilon_0}) , \label{KN approximate for derivative of mathcal A}\\
       & \frac{\partial^n}{\partial t^n} \cA(z,\tau,t) \ll_n \lambda^{-1}(\beta^{1/2-\epsilon_0})^{n+1},\qquad\text{ for any integer }n\geq2.\label{KN approximate for higher derivative of mathcal A}
     \end{align}
\end{proposition}
\begin{proof}
    By applying $\partial/\partial t$ to \eqref{eq: dist function}, we get
    \begin{align*}
        \sinh(\cA(z,\tau,t))\frac{\partial}{\partial t}\cA(z,\tau,t) = \sinh(t)-\frac{1}{2}e^{-t}(|z|^4+2|z|^2+\tau^2)=\sinh(t)+O(\lambda^{-1}).
    \end{align*}
    Note that
    \begin{align*}
        \sinh(\cA(z,\tau,t))=\sqrt{\cosh^2(\cA(z,\tau,t))-1}=\sqrt{\sinh^2(t)+O(\lambda^{-1})}=|\sinh(t)|(1+O(\lambda^{-1}\sinh^{-2}(t))),
    \end{align*}
    so
    \begin{align*}
        \frac{\partial}{\partial t}\cA(z,\tau,t)=\sgn(\sinh(t))+O(\lambda^{-1}\sinh^{-2}(t))=\sgn t+O(\lambda^{-1}\beta^{1-2\epsilon_0}),
    \end{align*}
    where the last step is due to Lemma \ref{KN lower bound for t from lower bound for A(z,t)}.

    To bound the higher derivatives, we note that
    \begin{align*}
        \frac{\partial}{\partial t}\mathcal{A}(z,\tau,t)=\frac{\sinh(t)-\frac{1}{2}e^{-t}(|z|^4+2|z|^2+\tau^2)}{\sqrt{\left(\cosh(t)+\frac{1}{2}e^{-t}(|z|^4+2|z|^2+\tau^2)+|z|^2\right)^2-1}}
    \end{align*}
    and for any $w\in\BC$ with $|w|=O(1)$ the real part inside the square root is
    \begin{multline}\label{eq: real part in sqrt}
        \Re\left(\left(\cosh(w)+\frac{1}{2}e^{-w}(|z|^4+2|z|^2+\tau^2)+|z|^2\right)^2-1\right)\\
        =\Re\left(\sinh^2(w)+O(\lambda^{-1})\right)=\frac{\cosh(2\Re(w))\cos(2\Im (w))-1}{2}+O(\lambda^{-1})\\
    = \sinh^2(\Re(w)) +O(\Im(w)^2)+O(\lambda^{-1}).
    \end{multline}
    For $t\ll1$ given in the proposition, by Lemma \ref{KN lower bound for t from lower bound for A(z,t)}, we have $|t|\gg \beta^{-1/2+\epsilon_0}$. Therefore, for $C>0$ and for any $w\in\BC$ with $|w-t|\leq C \beta^{-1/2+\epsilon_0}$, \eqref{eq: real part in sqrt} is
    \[
    \sinh^2(\Re(w))+O((C\beta^{-1/2+\epsilon_0})^2)+O(\lambda^{-1}) \gg\beta^{-1+2\epsilon_0},
    \]
    which holds when $C>0$ is chosen to be sufficiently small. Thus, 
    \[\frac{\partial}{\partial w}\mathcal{A}(z,w)=\frac{\sinh(w)-\frac{1}{2}e^{-w}(|z|^4+2|z|^2+\tau^2)}{\sqrt{\left(\cosh(w)+\frac{1}{2}e^{-w}(|z|^4+2|z|^2+\tau^2)+|z|^2\right)^2-1}}\]
    can be extended to be holomorphic (in $w$) on the ball of radius $C\beta^{-1/2+\epsilon_0}$ centered at $t$. By applying Cauchy's integral formula, with the bound 
    \(\frac{\partial}{\partial w}\cA(z,\tau,w)-\sgn t=O(\lambda^{-1}\beta^{1-2\epsilon_0})\) for $|w-t|\leq C \beta^{-1/2+\epsilon_0}$, we obtain, for any $n\geq1$,
    \begin{align*}
        \frac{\partial^{n+1}}{\partial t^{n+1}} \cA(z,\tau,t)=\frac{\partial^{n}}{\partial t^{n}} \left(\frac{\partial}{\partial t}\cA(z,\tau,t)-\sgn t\right) \ll_n \lambda^{-1}\beta^{1-2\epsilon_0}(C\beta^{-1/2+\epsilon_0})^{-n},
    \end{align*}
    which proves \eqref{KN approximate for higher derivative of mathcal A}.
\end{proof}

\subsubsection{Proof of \eqref{eq:equa bd for J2}}
We apply asymptotics for the spherical function from \cite[Theorem 1.5]{marshall2016p}.
There are functions $f_{\pm} \in C^\infty ((0,3)\times \BR) $ such that
\begin{align}\label{KN bound for f pm}
    \frac{\partial^n}{\partial t^n} f_\pm(t,s) \ll_n t^{-n}(st)^{-3/2}
\end{align}
and
\begin{align}\label{KN asymtotic for spherical function}
    \varphi_s (a(t)) = f_+(t,s)e^{ist}  + f_-(t,s)e^{-ist}  + O_A((st)^{-A})
\end{align}
for $t\in(0,3)$. Plug the formula \eqref{KN asymtotic for spherical function} into the integral $J_2(\phi,s)$. Because of the support of $b_2$, the error term in \eqref{KN asymtotic for spherical function} contributes $O_A(s^{-A})$, which may be ignored. We shall only consider the integral involving $f_+$, as the integral involving $f_-$ is similar.
By expanding one $\phi$ via the Fourier inversion and replacing $t_2$ with $t_2+t_1$, the integral we are considering is equal to
\begin{align}
    \frac{1}{2\pi}\iint_\BC\iint_\BR&\overline{b\phi(z_1,\tau_1,t_1)}\sF_t\phi(z_2,\tau_2,s_2)16e^{is_2t_1-4t_1}\notag\\
    &\left(\int_0^\infty\chi(t_2)b_2(\cA(z,\tau,t_2)) f_+(\cA(z,\tau,t_2),s)\exp(is\cA(z,\tau,t_2)+is_2t_2)dt_2\right.\label{KN main inner integral away from spec}\\
    +&\left.\int_{-\infty}^0 \chi(t_2)b_2(\cA(z,\tau,t_2)) f_+(\cA(z,\tau,t_2),s)\exp(is\cA(z,\tau,t_2)+is_2t_2)dt_2\right)\label{KN main inner integral away from spec 2}\\
    &dz_1 dz_2 d\tau_1d\tau_2 dt_1 ds_2.\notag
\end{align}
Here $|z|,|\tau|\ll\lambda^{-1/2}$ are defined by \eqref{eq: defn for z and tau}. We denote by $\chi(t_2)=b(z_2,t_2+t_1)e^{-2 t_2}$ and omit the variables $z_2,t_1$ as the estimates for the derivatives of $\chi$ with respect to $t_2$ are uniform on $z_2,t_1$. If we replace $t_2$ with $\beta^{-1/2+\epsilon_0}t_2$, then we write the first inner integral \eqref{KN main inner integral away from spec} in the form $\int_0^\infty \alpha(t_2) e^{i(s+s_2)^{1/2+\epsilon_0} f(t_2)} dt_2$ with
\begin{align*}
    &\alpha(t_2)= \beta^{-1/2+\epsilon_0}\chi(\beta^{-1/2+\epsilon_0}t_2)b_2(\cA(z,\tau,\beta^{-1/2+\epsilon_0}t_2)) f_+(\cA(z,\tau,\beta^{-1/2+\epsilon_0}t_2),s),\\
    &f(t_2)=\frac{s}{(s+s_2)^{1/2+\epsilon_0}}\cA(z,\tau,\beta^{-1/2+\epsilon_0}t_2)+\frac{s_2}{(s+s_2)^{1/2+\epsilon_0}}\beta^{-1/2+\epsilon_0}t_2.
\end{align*}
By \eqref{KN approximate for derivative of mathcal A} and the Fourier support assumption on $\phi$: $|s_2\pm \lambda|\geq\beta$, we have the following estimate for the derivative of the phase function $f(t_2)$:
    \begin{align*}
        \left|\frac{\partial f}{\partial t_2}\right| = |s+s_2|^{-1/2-\epsilon_0}\beta^{-1/2+\epsilon_0}\left|s(1 + O(\lambda^{-1}\beta^{1- 2\epsilon_0}))+s_2\right| = \left(\frac{|s+s_2|}{\beta} \right)^{1/2-\epsilon_0}+ O(\beta^{-2\epsilon_0})\gg1.
    \end{align*}
For any $n\geq2$, \eqref{KN approximate for higher derivative of mathcal A} implies that,
    \begin{align*}
        \frac{\partial^nf}{\partial t_2^n}&\ll_n \frac{s}{|s+s_2|^{1/2+\epsilon_0}} \lambda^{-1}(\beta^{1/2-\epsilon_0})^{n+1}(\beta^{-1/2+\epsilon_0})^n\ll \beta^{-2\epsilon_0}\ll1.
    \end{align*}
By \eqref{KN bound for f pm}, if $t\gg \beta^{-1/2+\epsilon_0}$ and $n\geq0$, then
\begin{align*}
    \frac{\partial^n}{\partial t^n} b_2(t)f_+(t,s) \ll_n \lambda^{-3/2}(\beta^{1/2-\epsilon_0})^{3/2}(\beta^{1/2-\epsilon_0})^n=\lambda^{-3/2} (\beta^{1/2-\epsilon_0})^{n+3/2} .
\end{align*}
Moreover, by \eqref{KN approximate for derivative of mathcal A} and \eqref{KN approximate for higher derivative of mathcal A}, for any $n\geq1$, it may be seen that
\begin{align*}
    \frac{\partial^n}{\partial t_2^n}\cA(z,\tau,\beta^{-1/2+\epsilon_0}t_2)\ll (1+\lambda^{-1}(\beta^{1/2-\epsilon_0})^{n+1})(\beta^{-1/2+\epsilon_0})^{n}\ll \beta^{-1/2+\epsilon_0}\ll1.
\end{align*}
Therefore, by the chain rule to higher derivatives, for any $n\geq0$ we obtain that the $n$-th derivative of the amplitude function $\alpha(t_2)$ is bounded by
    \begin{align*}
        \frac{\partial^n\alpha}{\partial t_2^n}\ll_n\beta^{-1/2+\epsilon_0}\lambda^{-3/2} (\beta^{1/2-\epsilon_0})^{n+3/2}=\lambda^{-3/2}(\beta^{1/2-\epsilon_0})^{n+1/2}.
    \end{align*}
Integrating by parts $n$ times gives that the integral \eqref{KN main inner integral away from spec}, which has been written as the form of the oscillatory integral $\int_0^\infty \alpha(t_2) e^{i(s+s_2)^{1/2+\epsilon_0} f(t_2)} dt_2$, is bounded by
\begin{align*}
    \ll_n& |s+s_2|^{-n(1/2+\epsilon_0)}\int_0^\infty\left|\left(\frac{\partial}{\partial t_2}\frac{1}{\partial f/\partial t_2}\right)^n\alpha\right| dt_2\\
    \ll_n& |s+s_2|^{-n(1/2+\epsilon_0)}\lambda^{-3/2}(\beta^{1/2-\epsilon_0})^{n+1/2}\beta^{1/2-\epsilon_0}\ll |s+s_2|^{-n\epsilon_0}.
\end{align*}
In summary, \eqref{KN main inner integral away from spec} is $\ll_{A,\epsilon_0}|s+s_2|^{-A}$. The second inner integral \eqref{KN main inner integral away from spec 2} can be bounded in the same way by replacing $s+s_2$ with $s-s_2$ so \eqref{KN main inner integral away from spec 2} is $\ll_{A,\epsilon_0}|s-s_2|^{-A}$. Hence, with this estimate and the fact $|s\pm s_2|\gg\beta$, by applying the Cauchy-Schwarz inequality and the Plancherel theorem, we obtain \eqref{eq:equa bd for J2}.


\bibliographystyle{plain} 
\bibliography{refs.bib} 

@article{iwaniec1995norms,
  AUTHOR = {Iwaniec, H. and Sarnak, P.},
     TITLE = {{$L^\infty$} norms of eigenfunctions of arithmetic surfaces},
   JOURNAL = {Ann. of Math. (2)},
  FJOURNAL = {Annals of Mathematics. Second Series},
    VOLUME = {141},
      YEAR = {1995},
    NUMBER = {2},
     PAGES = {301--320},
      ISSN = {0003-486X,1939-8980},
   MRCLASS = {11F72 (11F37 58G25 81Q50)},
  MRNUMBER = {1324136},
MRREVIEWER = {Jens\ Bolte},
       DOI = {10.2307/2118522},
       URL = {https://doi.org/10.2307/2118522},
}

@article{selberg1956harmonic,
  AUTHOR = {Selberg, A.},
     TITLE = {Harmonic analysis and discontinuous groups in weakly symmetric
              {R}iemannian spaces with applications to {D}irichlet series},
   JOURNAL = {J. Indian Math. Soc. (N.S.)},
  FJOURNAL = {The Journal of the Indian Mathematical Society. New Series},
    VOLUME = {20},
      YEAR = {1956},
     PAGES = {47--87},
      ISSN = {0019-5839,2455-6475},
   MRCLASS = {10.1X},
  MRNUMBER = {88511},
MRREVIEWER = {F.\ V.\ Atkinson},
}

@book{helgason1984groups,
  AUTHOR = {Helgason, Sigurdur},
     TITLE = {Groups and geometric analysis},
    SERIES = {Mathematical Surveys and Monographs},
    VOLUME = {83},
      NOTE = {Integral geometry, invariant differential operators, and
              spherical functions,
              Corrected reprint of the 1984 original},
 PUBLISHER = {American Mathematical Society, Providence, RI},
      YEAR = {2000},
     PAGES = {xxii+667},
      ISBN = {0-8218-2673-5},
   MRCLASS = {22-02 (22E30 22E46 43A85 43A90 44A12 53-02 58-02)},
  MRNUMBER = {1790156},
       DOI = {10.1090/surv/083},
       URL = {https://doi.org/10.1090/surv/083},
}

@book{helgason1994geometric,
  AUTHOR = {Helgason, Sigurdur},
     TITLE = {Geometric analysis on symmetric spaces},
    SERIES = {Mathematical Surveys and Monographs},
    VOLUME = {39},
   EDITION = {Second},
 PUBLISHER = {American Mathematical Society, Providence, RI},
      YEAR = {2008},
     PAGES = {xviii+637},
      ISBN = {978-0-8218-4530-1},
   MRCLASS = {22E46 (22E30 43A85 53C35 58J60)},
  MRNUMBER = {2463854},
MRREVIEWER = {Jacques\ Faraut},
       DOI = {10.1090/surv/039},
       URL = {https://doi.org/10.1090/surv/039},
}

@article{marshall2016p,
  AUTHOR = {Marshall, Simon},
     TITLE = {{$L^p$} norms of higher rank eigenfunctions and bounds for
              spherical functions},
   JOURNAL = {J. Eur. Math. Soc. (JEMS)},
  FJOURNAL = {Journal of the European Mathematical Society (JEMS)},
    VOLUME = {18},
      YEAR = {2016},
    NUMBER = {7},
     PAGES = {1437--1493},
      ISSN = {1435-9855,1435-9863},
   MRCLASS = {43A15 (22E30 35B45 35R01 58J50)},
  MRNUMBER = {3506604},
MRREVIEWER = {E.\ K.\ Narayanan},
       DOI = {10.4171/JEMS/619},
       URL = {https://doi.org/10.4171/JEMS/619},
}

@article{marshall2016geodesic,
  AUTHOR = {Marshall, Simon},
     TITLE = {Geodesic restrictions of arithmetic eigenfunctions},
   JOURNAL = {Duke Math. J.},
  FJOURNAL = {Duke Mathematical Journal},
    VOLUME = {165},
      YEAR = {2016},
    NUMBER = {3},
     PAGES = {463--508},
      ISSN = {0012-7094,1547-7398},
   MRCLASS = {11F25 (11F41 35P20)},
  MRNUMBER = {3466161},
MRREVIEWER = {Neven\ Grbac},
       DOI = {10.1215/00127094-3166736},
       URL = {https://doi.org/10.1215/00127094-3166736},
}

@article{marshall2015restrictions,
  AUTHOR = {Marshall, Simon},
     TITLE = {Restrictions of {$SL_3$} {M}aass forms to maximal flat
              subspaces},
   JOURNAL = {Int. Math. Res. Not. IMRN},
  FJOURNAL = {International Mathematics Research Notices. IMRN},
      YEAR = {2015},
    NUMBER = {16},
     PAGES = {6988--7015},
  FJOURNAL = {International Mathematics Research Notices. IMRN},
      ISSN = {1073-7928,1687-0247},
   MRCLASS = {11F60},
  MRNUMBER = {3428953},
MRREVIEWER = {Neven\ Grbac},
       DOI = {10.1093/imrn/rnu155},
       URL = {https://doi.org/10.1093/imrn/rnu155},
}

@article{blomer2019sup,
  AUTHOR = {Blomer, Valentin and Harcos, Gergely and Maga, P\'eter and
              Mili\'cevi\'c, Djordje},
     TITLE = {The sup-norm problem for {$\rm GL(2)$} over number fields},
   JOURNAL = {J. Eur. Math. Soc. (JEMS)},
  FJOURNAL = {Journal of the European Mathematical Society (JEMS)},
    VOLUME = {22},
      YEAR = {2020},
    NUMBER = {1},
     PAGES = {1--53},
      ISSN = {1435-9855,1435-9863},
   MRCLASS = {11F72 (11F25 11F55)},
  MRNUMBER = {4046009},
MRREVIEWER = {Martin\ Raum},
       DOI = {10.4171/jems/916},
       URL = {https://doi.org/10.4171/jems/916},
}

@article{blomer2016subconvexity,
  AUTHOR = {Blomer, Valentin and Maga, P\'eter},
     TITLE = {Subconvexity for sup-norms of cusp forms on {$\rm {PGL}(n)$}},
   JOURNAL = {Selecta Math. (N.S.)},
  FJOURNAL = {Selecta Mathematica. New Series},
    VOLUME = {22},
      YEAR = {2016},
    NUMBER = {3},
     PAGES = {1269--1287},
      ISSN = {1022-1824,1420-9020},
   MRCLASS = {11F55 (11D75 11F72)},
  MRNUMBER = {3518551},
MRREVIEWER = {Martin\ Raum},
       DOI = {10.1007/s00029-015-0219-5},
       URL = {https://doi.org/10.1007/s00029-015-0219-5},
}

@article{blomer2016sup,
  AUTHOR = {Blomer, Valentin and Pohl, Anke},
     TITLE = {The sup-norm problem on the {S}iegel modular space of rank
              two},
   JOURNAL = {Amer. J. Math.},
  FJOURNAL = {American Journal of Mathematics},
    VOLUME = {138},
      YEAR = {2016},
    NUMBER = {4},
     PAGES = {999--1027},
      ISSN = {0002-9327,1080-6377},
   MRCLASS = {58J50 (11F46 22E46 32N99)},
  MRNUMBER = {3538149},
MRREVIEWER = {Benjamin\ Linowitz},
       DOI = {10.1353/ajm.2016.0032},
       URL = {https://doi.org/10.1353/ajm.2016.0032},
}

@article{sogge1988concerning,
  AUTHOR = {Sogge, Christopher D.},
     TITLE = {Concerning the {$L^p$} norm of spectral clusters for
              second-order elliptic operators on compact manifolds},
   JOURNAL = {J. Funct. Anal.},
  FJOURNAL = {Journal of Functional Analysis},
    VOLUME = {77},
      YEAR = {1988},
    NUMBER = {1},
     PAGES = {123--138},
      ISSN = {0022-1236},
   MRCLASS = {35P05 (35J25 58G25)},
  MRNUMBER = {930395},
MRREVIEWER = {David\ Gurarie},
       DOI = {10.1016/0022-1236(88)90081-X},
       URL = {https://doi.org/10.1016/0022-1236(88)90081-X},
}

@article{blair2017refined,
  AUTHOR = {Blair, Matthew D. and Sogge, Christopher D.},
     TITLE = {Refined and microlocal {K}akeya-{N}ikodym bounds of
              eigenfunctions in higher dimensions},
   JOURNAL = {Comm. Math. Phys.},
  FJOURNAL = {Communications in Mathematical Physics},
    VOLUME = {356},
      YEAR = {2017},
    NUMBER = {2},
     PAGES = {501--533},
      ISSN = {0010-3616,1432-0916},
   MRCLASS = {58J50},
  MRNUMBER = {3707332},
MRREVIEWER = {Yuri\ A.\ Kordyukov},
       DOI = {10.1007/s00220-017-2977-8},
       URL = {https://doi.org/10.1007/s00220-017-2977-8},
}

@article{blair2018concerning,
   AUTHOR = {Blair, Matthew D. and Sogge, Christopher D.},
     TITLE = {Concerning {T}oponogov's theorem and logarithmic improvement
              of estimates of eigenfunctions},
   JOURNAL = {J. Differential Geom.},
  FJOURNAL = {Journal of Differential Geometry},
    VOLUME = {109},
      YEAR = {2018},
    NUMBER = {2},
     PAGES = {189--221},
      ISSN = {0022-040X,1945-743X},
   MRCLASS = {58J51 (35P15 35R01 42B37)},
  MRNUMBER = {3807318},
MRREVIEWER = {Anton\ Deitmar},
       DOI = {10.4310/jdg/1527040871},
       URL = {https://doi.org/10.4310/jdg/1527040871},
}

@article{avakumovic1956eigenfunktionen,
    AUTHOR = {Avakumovi\'c, Vojislav G.},
     TITLE = {\"Uber die {E}igenfunktionen auf geschlossenen {R}iemannschen
              {M}annigfaltigkeiten},
   JOURNAL = {Math. Z.},
  FJOURNAL = {Mathematische Zeitschrift},
    VOLUME = {65},
      YEAR = {1956},
     PAGES = {327--344},
      ISSN = {0025-5874,1432-1823},
   MRCLASS = {35.0X},
  MRNUMBER = {80862},
MRREVIEWER = {S.\ Bochner},
       DOI = {10.1007/BF01473886},
       URL = {https://doi.org/10.1007/BF01473886},
}

@article{levitan1952asymptotic,
    AUTHOR = {Levitan, B. M.},
     TITLE = {On the asymptotic behavior of the spectral function of a
              self-adjoint differential equation of the second order},
   JOURNAL = {Izv. Akad. Nauk SSSR Ser. Mat.},
  FJOURNAL = {Izvestiya Akademii Nauk SSSR. Seriya Matematicheskaya},
    VOLUME = {16},
      YEAR = {1952},
     PAGES = {325--352},
      ISSN = {0373-2436},
   MRCLASS = {36.0X},
  MRNUMBER = {58067},
MRREVIEWER = {E.\ A.\ Coddington},
}

@article{berard1977wave,
    AUTHOR = {B\'erard, Pierre H.},
     TITLE = {On the wave equation on a compact {R}iemannian manifold
              without conjugate points},
   JOURNAL = {Math. Z.},
  FJOURNAL = {Mathematische Zeitschrift},
    VOLUME = {155},
      YEAR = {1977},
    NUMBER = {3},
     PAGES = {249--276},
      ISSN = {0025-5874,1432-1823},
   MRCLASS = {58G99 (58E05)},
  MRNUMBER = {455055},
MRREVIEWER = {P.\ G\"unther},
       DOI = {10.1007/BF02028444},
       URL = {https://doi.org/10.1007/BF02028444},
}

@article{hassell2015improvement,
    AUTHOR = {Hassell, Andrew and Tacy, Melissa},
     TITLE = {Improvement of eigenfunction estimates on manifolds of
              nonpositive curvature},
   JOURNAL = {Forum Math.},
  FJOURNAL = {Forum Mathematicum},
    VOLUME = {27},
      YEAR = {2015},
    NUMBER = {3},
     PAGES = {1435--1451},
      ISSN = {0933-7741,1435-5337},
   MRCLASS = {35R01 (35P15 53C21 58J50)},
  MRNUMBER = {3341481},
       DOI = {10.1515/forum-2012-0176},
       URL = {https://doi.org/10.1515/forum-2012-0176},
}

@article{blair2019logarithmic,
    AUTHOR = {Blair, Matthew D. and Sogge, Christopher D.},
     TITLE = {Logarithmic improvements in {$L^p$} bounds for eigenfunctions
              at the critical exponent in the presence of nonpositive
              curvature},
   JOURNAL = {Invent. Math.},
  FJOURNAL = {Inventiones Mathematicae},
    VOLUME = {217},
      YEAR = {2019},
    NUMBER = {2},
     PAGES = {703--748},
      ISSN = {0020-9910,1432-1297},
   MRCLASS = {58J50 (35P15 35R01 58J05 81Q20)},
  MRNUMBER = {3987179},
MRREVIEWER = {Akira\ Asada},
       DOI = {10.1007/s00222-019-00873-6},
       URL = {https://doi.org/10.1007/s00222-019-00873-6},
}

@article {BS15APDE,
    AUTHOR = {Blair, Matthew D. and Sogge, Christopher D.},
     TITLE = {Refined and microlocal {K}akeya-{N}ikodym bounds for
              eigenfunctions in two dimensions},
   JOURNAL = {Anal. PDE},
  FJOURNAL = {Analysis \& PDE},
    VOLUME = {8},
      YEAR = {2015},
    NUMBER = {3},
     PAGES = {747--764},
      ISSN = {2157-5045,1948-206X},
   MRCLASS = {58J50 (35J25 35P15 35P20 35R01 42B37)},
  MRNUMBER = {3353830},
MRREVIEWER = {Leonid\ Friedlander},
       DOI = {10.2140/apde.2015.8.747},
       URL = {https://doi.org/10.2140/apde.2015.8.747},
}

@article {Sogge11Tohoku,
    AUTHOR = {Sogge, Christopher D.},
     TITLE = {Kakeya-{N}ikodym averages and {$L^p$}-norms of eigenfunctions},
   JOURNAL = {Tohoku Math. J. (2)},
  FJOURNAL = {The Tohoku Mathematical Journal. Second Series},
    VOLUME = {63},
      YEAR = {2011},
    NUMBER = {4},
     PAGES = {519--538},
      ISSN = {0040-8735,2186-585X},
   MRCLASS = {58J51 (35L20 42B37)},
  MRNUMBER = {2872954},
MRREVIEWER = {Elena\ A.\ Mazepa},
       DOI = {10.2748/tmj/1325886279},
       URL = {https://doi.org/10.2748/tmj/1325886279},
}

@incollection {Bourgain,
    AUTHOR = {Bourgain, J.},
     TITLE = {Geodesic restrictions and {$L^p$}-estimates for eigenfunctions
              of {R}iemannian surfaces},
 BOOKTITLE = {Linear and complex analysis},
    SERIES = {Amer. Math. Soc. Transl. Ser. 2},
    VOLUME = {226},
     PAGES = {27--35},
 PUBLISHER = {Amer. Math. Soc., Providence, RI},
      YEAR = {2009},
      ISBN = {978-0-8218-4801-2; 0-8218-4801-1},
   MRCLASS = {58J50 (58J37)},
  MRNUMBER = {2500507},
MRREVIEWER = {Julie\ Rowlett},
       DOI = {10.1090/trans2/226/03},
       URL = {https://doi.org/10.1090/trans2/226/03},
}

@article{gao2025sharp,
  title={{Sharp microlocal Kakeya--Nikodym estimates for H\"ormander operators and spectral projectors}},
  author={Gao, Chuanwei and Wu, Shukun and Xi, Yakun},
  journal={arXiv preprint arXiv:2509.01116},
  year={2025}
}

@book{goldman1999complex,
  AUTHOR = {Goldman, William M.},
     TITLE = {Complex hyperbolic geometry},
    SERIES = {Oxford Mathematical Monographs},
      NOTE = {Oxford Science Publications},
 PUBLISHER = {The Clarendon Press, Oxford University Press, New York},
      YEAR = {1999},
     PAGES = {xx+316},
      ISBN = {0-19-853793-X},
   MRCLASS = {32Q45 (30F45 51M10 57M50)},
  MRNUMBER = {1695450},
MRREVIEWER = {John\ R.\ Parker},
}

@article{hou2024restrictions,
  title={{Restrictions of Maass forms on $\mathrm{SL}(2,\mathbb{C}) $ to hyperbolic surfaces and geodesic tubes}},
  author={Hou, Jiaqi},
  journal={arXiv preprint arXiv:2410.17164},
  year={2024}
}

\end{document}